\documentclass[a4paper,10pt]{amsart}
\usepackage{standalone}
\usepackage{amsfonts}
\usepackage{amsthm}
\usepackage{amssymb}
\usepackage{amsmath}
\usepackage{theoremref}
\usepackage{enumerate}
\usepackage{bbm}
\usepackage{bm}
\usepackage{pgfplots}
\pgfplotsset{compat=1.18} 
\usepgfplotslibrary{fillbetween}
\usetikzlibrary{intersections}
\usepackage{mathtools}
\usetikzlibrary{patterns}

\usepackage{geometry}
\usepackage{a4wide}

\numberwithin{equation}{section}

\newcommand{\A}{\mathcal{A}}
\renewcommand{\L}{\mathcal{L}}
\newcommand{\C}{\mathcal{C}}

\newcommand{\res}{\text{res}}
\newcommand{\co}{\text{core}}
\newcommand{\K}{\mathcal{K}}
\newcommand{\N}{\mathcal{N}}

\newcommand{\T}{\mathbb{T}}

\newcommand{\conj}[1]{\overline{#1}}
\newcommand{\D}{\mathbb{D}}

\newcommand{\R}{\mathbb{R}}
\renewcommand{\H}{\mathbb{H}}
\newcommand{\Po}{\mathcal{P}}
\newcommand{\cD}{\conj{\mathbb{D}}}
\newcommand{\ip}[2]{\big\langle #1, #2 \big\rangle}
\newcommand{\dist}[2]{\text{dist}( #1, #2 ) }

\newcommand{\m}{\textit{m}}

\newcommand{\Hb}{\mathcal{H}(b)}
\newcommand{\hb}{\mathcal{H}(b)}

\newcommand{\h}{\mathcal{H}}

\renewcommand\Re{\operatorname{Re}}
\renewcommand\Im{\operatorname{Im}}

\newtheorem{mainthm}{Theorem}

\newtheorem{thm}{Theorem}[section]
\newtheorem*{thm*}{Theorem}
\newtheorem{lem}[thm]{Lemma}

\newtheorem{cor}[thm]{Corollary}
\newtheorem*{cor*}{Corollary}
\newtheorem{prop}[thm]{Proposition}
\theoremstyle{definition}
\newtheorem{example}[thm]{Example}
\theoremstyle{definition}
\newtheorem{remark}[thm]{Remark}
\newtheorem{defn}[thm]{Definition}
\newtheorem{question}{Question}
\newtheorem{conjecture}{Conjecture}

\title[Shift operators, Cauchy integrals and approximations]{Shift operators, Cauchy integrals and approximations}
\author{Bartosz Malman}

\address{Division of Mathematics and Physics, 
        Mälardalen University,
		Västerås, Sweden}
\email{bartosz.malman@mdu.se} 

\begin{document}

\begin{abstract} This article consists of two connected parts. In the first part, we study the shift invariant subspaces in certain $\Po^2(\mu)$-spaces, which are the closures of analytic polynomials in the Lebesgue spaces $\L^2(\mu)$ defined by a class of measures $\mu$ living on the closed unit disk $\cD$. The measures $\mu$ which occur in our study have a part on the open disk $\D$ which is radial and decreases at least exponentially fast near the boundary. Our focus is on those shift invariant subspaces which are generated by a bounded function in $H^\infty$. In this context, our results are definitive. We give a characterization of the cyclic singular inner functions by an explicit and readily verifiable condition, and we establish certain permanence properties of non-cyclic ones which are important in the applications. The applications take up the second part of the article. We prove that if a function $g \in \L^1(\T)$ on the unit circle $\T$ has a Cauchy transform with Taylor coefficients of order $\mathcal{O}\big(\exp(-c \sqrt{n})\big)$ for some $c > 0$, then the set $U = \{x \in \T : |g(x)| > 0 \}$ is essentially open and $\log |g|$ is locally integrable on $U$. We establish also a simple characterization of analytic functions $b: \D \to \D$ with the property that the de Branges-Rovnyak space $\hb$ contains a dense subset of functions which, in a sense, just barely fail to have an analytic continuation to a disk of radius larger than 1. We indicate how close our results are to being optimal and pose a few questions.
\end{abstract}

\maketitle

\section{Introduction and main results}
\label{introsec}

\subsection{Some background} We will study spaces of analytic functions corresponding to Borel measures of the form \begin{equation}
    \label{MuGeneralStructureEq}
    d\mu(z) = G(1-|z|)\, dA(z) + w(z)d\m(z),
\end{equation} where $dA$ and $dm$ are the area and arc-length measures on, respectively, the unit disk $\D := \{ z \in \mathbb{C} : |z| < 1\}$ and its boundary circle $\T := \{ z \in \mathbb{C} : |z| = 1 \}$. The radial weight $G(1-|z|)$ living on $\D$ is defined in terms of a continuous, increasing and positive function $G$, and the weight $w$ living on $\T$ is a general Borel measurable non-negative integrable function. Given such a measure, we may construct first the Lebesgue space $\L^2(\mu)$ of (equivalence classes of) Borel measurable functions living on the carrier of $\mu$, and next consider its subspace $\Po^2(\mu)$, by which we denote the smallest closed subspace of $\L^2(\mu)$ which contains the set $\Po$ of analytic polynomials. The space $\Po^2(\mu)$ will be the setting for the first part of our study.

The \textit{shift} operator $M_z: \Po^2(\mu) \to \Po^2(\mu)$, which takes a function $f(z)$ to $zf(z)$, is a \textit{subnormal} operator, in the sense that it is the restriction of a normal operator, namely $M_z: \L^2(\mu) \to \L^2(\mu)$, to an invariant subspace. From the point of view of an operator theorist, the significance of the pair $(\Po^2(\mu), M_z)$ lies in the fact the study of subnormal operators can essentially be reduced to the study of the operator $M_z: \Po^2(\mu) \to \Po^2(\mu)$ for some measure $\mu$ which is compactly supported in the plane. The monograph \cite{conway1991theory} by Conway is an excellent source of information on this topic.

For measures such as \eqref{MuGeneralStructureEq}, the space $\Po^2(\mu)$ is, like $\L^2(\mu)$, a space of Borel measurable functions on the closed disk $\cD = \D \cup \T$. In certain cases it is even a space of \textit{analytic} functions on $\D$. In such a case, each element $f \in \Po^2(\mu)$, a priori interpreted as a function on $\cD$, has a unique restriction $f_\D$ to the disk $\D$. The restriction $f_\D$ must be an analytic function by the virtue of it being a locally uniform limit of analytic polynomials. We will below use the term \textit{irreducible} for such a space which is in this sense "analytic". It is a difficult problem (and in general open) to determine which weight pairs $(G,w)$ as in \eqref{MuGeneralStructureEq} produce an irreducible space. Khrushchev in the article \cite{khrushchev1978problem} solved certain special cases of the problem. For instance, his results apply to $G(t) = t^n$ for some $n > 0$, and $w = 1_E$ being a characteristic function of a set $E \subset \T$ in a certain class (defined in terms of Beurling-Carleson conditions). Already these results have fascinating applications to function theory, of which there are plenty in \cite{khrushchev1978problem}. The article \cite{malman2022thomson} builds on Khrushchev's work, explains the structure of $\Po^2(\mu)$ when $w = 1_E$ and $E$ is a general subset of $\T$, and showcases further applications to the theory of the Cauchy integral operator and de Branges-Rovnyak spaces.

\subsection{Irreducible $\Po^2(\mu)$-spaces} Recently, the author found in \cite{malman2023revisiting} an exact condition for irreducibility of $\Po^2(\mu)$ in the case when $G(t)$ decays at least exponentially as $t \to 0^+$, thus confirming a conjecture by Kriete and MacCluer from \cite{kriete1990mean}. Roughly speaking, if $G(t)$ is smaller than the weight $\exp(-c t^{-1})$ for some $c > 0$, or more precisely if 
\begin{equation} \label{ExpDecTag}
\liminf_{x \to 0^+} \, x \log 1/G(x) > 0, \tag{ExpDec}    
\end{equation} but large enough to satisfy 
\begin{equation} \label{LogLogIntTag}
\int_0^d \log \log (1/G(x)) \, dx < \infty, \tag{LogLogInt} \end{equation} for some $d > 0$, then the space $\Po^2(\mu)$ is irreducible if and only if the carrier set of the measure $d\mu_\T = w \,d\m$ on $\T$ can be covered by intervals $I$ satisfying the condition \begin{equation}
    \label{logwIntIntegr}
    \int_I \log w \, d\m > -\infty.
\end{equation}
In order to properly state the result we will need to define the following concept of core sets. For our purposes this concept is critical, and it will appear frequently throughout the article. 

\begin{defn}{\textbf{(Core sets of weights)}} \thlabel{CoreDef} Let $w$ be a non-negative integrable function on $\T$. We define $\co (w)$ to be the union of all open intervals $I$ for which \eqref{logwIntIntegr} holds. In other words, 
\begin{equation}
    \co (w) = \{ x \in \T : \text{ there exists open } I \text{ containing }
x \text{ for which } \eqref{logwIntIntegr} \text{ holds } \}
\end{equation}
\end{defn}
The set $\co (w)$ is open, and it does not depend on the particular representative of $w$ in the space $\L^1(\T)$ of equivalence classes of functions which are Lebesgue integrable on $\T$ with respect to $d\m$.

\begin{defn}{\textbf{(Carrier sets)}} \label{CarrierDef}
Let $\eta$ be a non-negative Borel measure on $\T$. A Borel subset $E$ of $\T$ is a \textit{carrier} for $\eta$ if \[\eta( \T \setminus E) = 0. \]
If $w$ is a Borel measurable function on $\T$, then we say that a set $E$ is a carrier for $w$ if it is a carrier for the Borel measure $w\, d\m$. 
\end{defn}
Carriers are obviously not unique. The set
\begin{equation}
    \label{CarrierWEq} \{ x \in \T : w(x) > 0 \}
\end{equation} is a carrier for $w$. If $w$ is only defined up to a set of $m$-measure zero, then we may take as a carrier for $w$ any set differing from \eqref{CarrierWEq} by a set of $m$-measure zero. Since $\log 0 = -\infty$, it is obvious from \eqref{logwIntIntegr} that $\co (w)$ is essentially contained in any carrier of $w$. 

Irreducibility of $\Po^2(\mu)$-spaces of the form \eqref{MuGeneralStructureEq} with $G$ satisfying \eqref{ExpDecTag} and \eqref{LogLogIntTag} can be characterized in terms of core sets. The next theorem, fundamental to our study, follows from \cite[Theorem A]{malman2023revisiting}, with the non-trivial part being the equivalence of the third condition and the other two.

\begin{thm} \thlabel{IrrDef} For a space $\Po^2(\mu)$ defined by a measure $\mu$ of the form \eqref{MuGeneralStructureEq}, with $G$ satisfying \eqref{ExpDecTag} and \eqref{LogLogIntTag}, the following three conditions are equivalent:

\begin{enumerate}[(i)]
    \item the space $\Po^2(\mu)$ contains no non-trivial characteristic function of a measurable subset of $\cD$: if $A$ is a Borel subset of $\cD$ and $1_A \in \Po^2(\mu)$ is not the zero element, then $1_A = 1_{\cD}$.
    \item the space $\Po^2(\mu)$ is a space of analytic functions on $\D$ in which the analytic polynomials are dense,
    \item the set $\co (w)$ is a carrier for $w$, or in other words it coincides with \eqref{CarrierWEq}, up to a set of $m$-measure zero.
\end{enumerate}
\end{thm}

\begin{defn}\thlabel{IrrDef2}{\textbf{(Irreducible spaces)}} A space $\Po^2(\mu)$ is \textit{irreducible} if it satisfies the three equivalent conditions stated in \thref{IrrDef}.    
\end{defn}

In particular, the following measures $\mu$ correspond to irreducible $\Po^2(\mu)$:

\begin{equation}
    \label{T1} \tag{$T1$}
    d\mu(z) =  \exp\Big(- \frac{c}{(1-|z|)^\beta}\Big) dA(z) + w(z) d\m(z), \quad c > 0, \, \beta \geq 1
\end{equation}
and 
\begin{equation}
    \label{T2} \tag{$T2$}
    d\mu(z) = \exp\Bigg(- c \exp\Bigg( \frac{1}{(1-|z|)^{\alpha}} \Bigg) \Bigg) dA(z) + w(z) d\m(z), \quad c > 0, \, \alpha \in (0,1).
\end{equation}

If the $\co (w)$ is not a carrier of $w$, then the space $\Po^2(\mu)$ will contain a full Lebesgue space $\L^2(w_r d\m)$, members of which live only on $\T$. Here $w_r$ denotes a certain \textit{residual} weight. The residuals play no role in the statements of our main results, but will be important in the proofs. Their definition is postponed to coming sections.
%The present article in essence concerns itself with further development of the theory, and various applications, of the spaces appearing in \thref{IrrDef}. Measures of the form \eqref{T1}, and the irreducibility of the corresponding space, will be used decisively in proofs of several of our main results pretaining to the Cauchy integral operator and de Branges-Rovnyak spaces $\hb$ (introduced below). The measures \eqref{T2} are good to keep in mind as simple examples in which various computations can be readily carried out and to which our irreducibility theory applies. A large chunk of the article is concerned with the study $M_z$-invariant subspaces of irreducible $\Po^2(\mu)$. These results will also find an application in $\hb$-theory.

The reader might wonder what happens in the case $\beta < 1$ in \eqref{T1}. Then \eqref{ExpDecTag} is violated, and condition $(iii)$ in \thref{IrrDef} implies $(ii)$, but the converse is false. This can be inferred from work of Khruschev in \cite{khrushchev1978problem}, and this idea is further elaborated on in \cite{malman2022thomson}. Also one might ask what happens if $\alpha \geq 1$ in \eqref{T2}, which means that \eqref{LogIntTag} is violated. This case is less interesting: Volberg's theorem in \cite{vol1987summability} implies that $\Po^2(\mu)$ is then either a close cousin of the Hardy space $H^2$ (this happens when $\int_\T \log w \,d\m > -\infty$) or it is not a space of analytic functions at all (if $\int_\T \log w \, d\m = -\infty$). See also the introductory section to \cite{malman2023revisiting} for a more detailed account.

\subsection{Invariant subspaces generated by singular inner functions}
Having established fairly sharp conditions for irreducibility, a way opens to an operator and function theoretic study of this class of spaces. Motivated by certain applications which will soon be detailed, in the first part of the article we study the structure of $M_z$-invariant subspaces of $\Po^2(\mu)$ generated by functions in $H^\infty$, the algebra of bounded analytic functions in $\D$. This question readily reduces to the study of invariant subspaces generated by \textit{singular inner functions}
\begin{equation}
    \label{SingInnerDef}
    S_\nu(z) = \exp \Big( -\int_\T \frac{x + z}{x - z} d\nu(x)\Big), \quad z \in \D,
\end{equation} where $\nu$ is a finite positive singular Borel measure on $\T$. For $h \in H^\infty$, we will denote by $[h]$ the smallest $M_z$-invariant subspace containing $h$. It is well-known that any singular inner function generates a non-trivial invariant subspace in the classical Hardy space $H^2$ of square-summable Taylor series, and it is almost as well-known that in order for $S_\nu$ to generate a non-trivial invariant subspace in the standard weighted Bergman spaces (which are $\Po^2(\mu)$-spaces of the kind \eqref{MuGeneralStructureEq} themselves, with $G(t) = t^n$ for some $n > -1$, and $w \equiv 0$) we must have $\nu(A) > 0$ for some Beurling-Carleson set $A$ (see \cite{korenblum1975extension}, \cite{korenblum1977beurling}, \cite{roberts1985cyclic}). 

Our first main result characterizes the cyclic singular inner functions in the considered class of  $\Po^2(\mu)$-spaces. By cyclicity we mean that $[S_\nu] = \Po^2(\mu)$. It is not hard to see that the minimal considered rate of decay \eqref{ExpDecTag} of the part of $\mu$ living on $\D$ makes every non-vanishing bounded function be cyclic in $\Po^2(\mu)$ in the case that $w = 0$. Thus only properties of $w$ can stop $S_\nu$ from being cyclic.

\begin{mainthm} \thlabel{CyclicityMainTheorem}
Let $\Po^2(\mu)$ be an irreducible space defined by a measure $\mu$ of the form \eqref{MuGeneralStructureEq}. The following two statements are equivalent.
\begin{enumerate}[(i)]
    \item The singular inner function $S_\nu$ is cyclic in $\Po^2(\mu)$.
    \item The measure $\nu$ assigns no mass to the core of the weight $w$: \[ \nu\big(\co(w) \big) = 0.\]
\end{enumerate}
\end{mainthm}
Note that $\co (w)$ is open, and hence Borel measurable, so $\nu\big( \co (w) \big)$ makes perfect sense. 

\begin{example} \thlabel{PointMassExample} Let $\delta_a$ be a point mass at $a \in \T$, and 
\begin{equation}
    \label{Examplewdef}
    w(x) = \exp \Bigg(-\frac{1}{|x-1|} \Bigg), \quad x \in \T.
\end{equation} Then it is easy to check that \[ \co (w) = \T \setminus \{1 \}.\] Consequently, the singular inner function \[S_{\delta_a}(z) = \exp \Bigg(- \frac{a+z}{a-z} \Bigg), \quad z \in \D\] is cyclic, in the considered class of $\Po^2(\mu)$ constructed from $w$ appearing in \eqref{Examplewdef}, if and only if $a = 1$. 
\end{example}

Having settled the cyclicity question, we turn our attention to the invariant subspace $[S_\nu]$ generated by a singular inner function corresponding to a measure $\nu$ which places all its mass on the core: $\nu(\T) = \nu \big( \co (w) \big)$. In other words, $\co (w)$ is a carrier for $\nu$. A problem which arises in the theory of normalized Cauchy integrals and de Branges-Rovnyak spaces $\hb$ (to be discussed below) is to determine which functions are contained in the intersection $H^2 \cap [S_\nu]$, or sometimes in $\N^+ \cap [S_\nu]$, where $\N^+$ is the \textit{Smirnov class} of the disk $\D$ (see \cite{garnett} for precise definitions):
\[ \N^+ = \{ u/v : u, v \in H^\infty, \, v \text{ outer} \} \] 
In this context, we have the following result.

\begin{mainthm} \thlabel{PermanenceMainTheorem}
    Let $S_\nu$ be a singular inner function corresponding to a measure $\nu$ which satisfies \[ \nu(\T) = \nu \big( \co (w) \big).\] In an irreducible $\Po^2(\mu)$-space defined by a measure $\mu$ of the form \eqref{MuGeneralStructureEq}, the invariant subspace $[S_\nu]$ satisfies \[ [S_\nu] \cap \N^+ \subset S_\nu \N^+.\]
\end{mainthm}

In other words, if $f \in \N^+$ can be approximated by polynomial multiples of $S_\nu$ in the norm of $\Po^2(\mu)$, and $\nu$ places all of its mass on $\co (w)$, then $S_\nu$ appears in the inner-outer factorization of $f$. Under the additional assumption that $w$ is bounded, a simple argument will show that in fact $[S_\nu] \cap H^2 = S_\nu H^2$. In \cite{limani_malman_2023} and \cite{limani2023problem}, the feature of $S_\nu$ appearing in \thref{PermanenceMainTheorem} is called its \textit{permanence property}. It is obvious that a singular inner function satisfying the permanence property cannot be cyclic.

For the considered class of spaces, \thref{CyclicityMainTheorem} and \thref{PermanenceMainTheorem} completely determine the structure of $M_z$-invariant subspaces generated by bounded analytic functions. Indeed, it follows that if $h = BS_\nu U \in H^\infty$ is the inner-outer factorization of $h$ into a Blaschke product $B$, singular inner function $S_\nu$ and outer function $U$, then \[ [h] = [S_{\nu_w}],\] where $\nu_w$ is the restriction of the singular measure $\nu$ to the set $\co (w)$.

\subsection{Functions of rapid spectral decay and Cauchy integrals}

Irreducible spaces find applications in the theory of Cauchy integrals.

\begin{defn}{\textbf{(Functions of rapid spectral decay)}} \label{ESDpropDef} Let $f(z) = \sum_{n \geq 0} f_nz^n$ be an analytic function in $\D$. If the Taylor coefficients $\{f_n\}_{n \geq 0}$ decay so fast that for some $c > 0$ we have
\begin{equation}
    \label{ESDpropEq} \tag{RSD}
    \sup_{n \geq 0} \, |f_n|\exp\big(c \sqrt{n}) < \infty,
\end{equation} then we say that $f$ is a function of \textit{rapid spectral decay}. \end{defn} 

Trivial examples of functions $f$ satisfying \eqref{ESDpropEq} are the analytic polynomials, and functions which extend analytically to a larger disk $r\D = \{ z \in \mathbb{C} : |z| < r\}$, $r > 1$. In those cases, the limit in \eqref{ESDpropEq} is zero even when the term $\exp\big(c\sqrt{n}\big)$ in \eqref{ESDpropEq} is replaced by $\exp\big(cn^{\alpha}\big)$ for $\alpha < 1$. Conversely, if  $f$ has an analytic extension to a disk around the origin of radius larger than 1, then $|f_n| = \mathcal{O}\big(\exp(-cn)\big)$ for some $c > 0$. 

Let us assume that $\nu$ is a finite Borel measure for which the Cauchy integral \begin{equation}
    \label{CauchyIntUncertThm} \C_\nu(z) := \int_\T \frac{1}{1-\conj{x}z} d\nu(x), \quad z \in \D
\end{equation} is a function satisfying \eqref{ESDpropEq}. Can we say something about the nature of the measure $\nu$? The Cauchy integral $\C_\nu$ has a representation of the form \[ \C_\nu(z) = \sum_{n \geq 0} \nu_n z^n, \quad z \in \D\] where $\{\nu_n\}_{n \geq 0}$ is the sequence of Fourier coefficients of $\nu$ indexed by non-negative integers. The rest of the coefficients are annihilated under $\C$, and the condition \eqref{ESDpropEq} gives us no information about $\nu_n$ for $n < 0$. However, the following statement is a consequence of the irreducibility of spaces corresponding to measures of the form \eqref{T1}.

\begin{mainthm} \thlabel{UncertThmRSD}
Let $\nu$ be a finite Borel measure on $\T$, and assume that the Cauchy integral $\C_\nu$, given by \eqref{CauchyIntUncertThm}, satisfies \eqref{ESDpropEq}. Then the measure $\nu$ is absolutely continuous with respect to the Lebesgue measure $d\m$: \[d \nu = g \,  d\m, \quad g \in \L^1(\T),\] and there exists an open set $U$ which differs from
\begin{equation} \label{gCarrierThmC}\{x \in \T : |g(x)| > 0 \} \end{equation} only by a set of $\m$-measure zero, with the property that to each $x \in U$ there corresponds an interval $I \subset U$ containing $x$ for which we have \[ \int_I \log |g(x)| \, d\m(x) > -\infty.\]
\end{mainthm} 

The function $\log |g|$ is, in general, not integrable on the entire open set $U$ appearing in \thref{UncertThmRSD}.

In a way, \thref{UncertThmRSD} is similar to the classical theorem of brothers Riesz on structure of measures $\nu$ on $\T$ with vanishing positive Fourier coefficients. In our setting, the vanishing of the coefficients is replaced by a weaker condition of their rapid decay forced by the condition \eqref{ESDpropEq}. It should be noted that if we were to replace in \eqref{ESDpropEq} the term $\exp\big( c \sqrt{n} \big)$ by $\exp \big( c n^{\alpha} \big)$ for any $\alpha < 1/2$, and thus consider the weaker unilateral spectral decay condition \[ \sup_{n \geq 0} \, |\nu_n| \exp \big( c n^{\alpha} \big) < \infty, \] then a structural result for $\nu$ as in \thref{UncertThmRSD} does not hold: $d\nu = g \, d\m$ will still be absolutely continuous, but examples show that $g$ can be chosen so that the set in \eqref{gCarrierThmC} is closed and contains no interval. This follows from a related work of Khrushchev in \cite{khrushchev1978problem}. There should be room for a slight improvement of the result (see the discussion in Section \ref{conjecturesSec} below). We ought to mention also that Volberg in \cite{vol1987summability} found spectral decay conditions making the set in \eqref{gCarrierThmC} fill up the whole circle $\T$. We will return to both these works below.

\subsection{Condition \eqref{ESDpropEq} in de Branges-Rovnyak spaces}

In most classical Hilbert spaces of analytic functions in the unit disk, the family of functions which extend analytically to a larger disk forms a dense subset of the space. This is not the case in Hilbert spaces of normalized Cauchy integrals. These are the so-called \textit{model spaces} $K_\theta$, where $\theta$ is an inner function, and more generally the \textit{de Branges-Rovnyak spaces} $\hb$, where the symbol $b$ is any analytic self-map of the unit disk. There are several ways to define the space $\hb$, the easiest perhaps being by stating that it is the Hilbert space of analytic functions on $\D$ with a reproducing kernel of the form \[ k_b(\lambda, z) = \frac{1-\conj{b(\lambda)}b(z)}{1-\conj{\lambda} z}, \quad \lambda, z \in \D.\] Alternatively, we may realize it as the space of \textit{normalized} Cauchy integrals of functions $ \tilde{f} \in \L^2(\nu_b)$, given in the special case $b(0) = 0$ by the formula 

\begin{equation}
    \label{ClarkRepHb}
    f(z) = \big(1 - b(z) \big) \int_\T \frac{\tilde{f}(x)}{1-\conj{x}z} d\nu_b(x), \quad z \in \D.
\end{equation} Here $\nu_b$ is the \textit{Aleksandrov-Clark measure} of $b$, these two objects being related by the formula
\begin{equation}\label{ClarkFormulaForb}
    \Re \Bigg( \frac{1 + b(z)}{1-b(z)}\Bigg) = \int_\T \frac{1-|z|^2}{|x-z|^2} d\nu_b(x), \quad z \in \D.
\end{equation}
The \textit{normalization} refers to multiplication of the Cauchy integral in \eqref{ClarkRepHb} by the factor $1-b(z)$, which ensures that the product lands in $H^2$. It is well-known that model spaces $K_\theta = \h(\theta)$ correspond to the purely singular measures $\nu_\theta$ in \eqref{ClarkFormulaForb}. In fact, every positive finite Borel measure $\nu$ on $\T$ corresponds to a function $b = b_\nu$ through the formula \eqref{ClarkFormulaForb}. See \cite{cauchytransform} for more details. 

If $\theta$ is a singular inner function, then $\K_\theta$ will contain no functions which extend analytically across $\T$. Moreover, it is a consequence of deep results on cyclicity of singular inner functions of Beurling from \cite{beurling1964critical}, and also of more recent results of El-Fallah, Kellay and Seip from \cite{el2012cyclicity}, that in fact if $\theta$ is singular, then for any non-zero function $f(z) = \sum_{n \geq 0} f_nz^n \in \K_\theta$ and for any $c > 0$ it holds that \[ \sup_{n \to \infty} |f_n| \exp\big(c \sqrt{n}) = \infty. \] This fact is not as deep as the two results cited above which imply it, but it is needed in the proof of one of our main results. For this reason, we give an elementary proof in Section \ref{ExtremalDecModelSpaceSec}. We mention also that a characterization of density in $\K_\theta$ of functions in $\A^\infty$, the algebra of functions analytic in $\D$ with all derivatives extending continuously to $\cD$, has been established \cite{smoothdensektheta}.

The situation is more interesting, and much more difficult to handle, in the general class of $\hb$-spaces. It was proved long ago by Sarason that the set $\Po$ of analytic polynomials is contained and dense in $\hb$ if and only if $b$ is a non-extreme point of the unit ball of $H^\infty$, a condition characterized by 
\begin{equation}
    \label{bNonExtremeIntegral}
    \int_\T \log (\Delta_b)\, d\m > -\infty,
\end{equation} where \[\Delta_b := \sqrt{1-|b|^2}.\] In terms of core sets, this result can be stated as follows, and a proof can be found in \cite{sarasonbook}.

\begin{thm*}[\textbf{Sarason}] Let $b:\D \to \D$ be an analytic function. The following three statements are equivalent.

\begin{enumerate}[(i)]
    \item The analytic polynomials are dense in $\hb$.
    \item The function $b$ is a non-extreme point of the unit ball of $H^\infty$.
    \item We have the set equality $\co (\Delta_b) = \T$.
\end{enumerate}
\end{thm*}

Since these conditions are very restrictive, it is tempting to make an effort to capture a larger class of symbols $b$ for which $\hb$ contains a dense subset of functions in some nice regularity class which is strictly larger than $\Po$. The article \cite{limani2023problem} connects the approximation problem in $\hb$ with the structure of $M_z$-invariant subspaces of $\Po^2(\mu)$, and \cite{limani_malman_2023} refines the method to prove the density of $\A^\infty \cap \hb$ for a large class of symbols $b$. The method from \cite{limani2023problem} is very general and applies to a wide range of approximation problems in $\hb$. In particular, it applies to approximations by functions in the class \eqref{ESDpropEq}. Since our structural results in \thref{CyclicityMainTheorem} and \thref{PermanenceMainTheorem} are definitive, we can prove also a definitive result on existence and density of functions $f \in \hb$ which satisfy \eqref{ESDpropEq}. In fact, we will prove a much stronger (and optimal) result.

In order to state our result, we will need to quantify the spectral decay of a function $f$ by a condition of the type \eqref{ESDpropEq} but with $\exp\big(c\sqrt{n}\big)$ replaced by faster increasing sequences. To this end, we define below in \thref{AdmissibleSequenceDef} the \textit{admissible sequences} $M = \{M_n\}_{n \geq 0}$. These sequences are logarithmically convex (at least eventually, for large $n$) and are decreasing to zero at least as fast as $\exp\big(-c \sqrt{n}\big)$, but satisfy a condition of the form \[ \sum_{n \geq 0} \frac{\log 1/M_n}{1+n^2} < \infty \] which prohibits, for instance, their exponentially fast decay. 

\begin{example}The sequence defined by 
\begin{equation}
    \label{ndivlognpSeq}
    M_n = \exp \Bigg( - c \frac{n}{(\log(n) + 1)^p}\Bigg), \quad n \geq 1 
\end{equation}is admissible for every $p > 1$ and $c > 0$, but it is not admissible for $p = 1$ and any $c > 0$.
\end{example}

\begin{mainthm} \thlabel{MainTheoremHbExistenceESD} Let $b:\D \to \D$ be an analytic function. The following three statements are equivalent.
\begin{enumerate}[(i)]
    \item The space $\hb$ contains a non-zero function $f$ which satisfies \eqref{ESDpropEq}.
    \item For any admissible sequence $\{M_n\}_{n \geq 0}$, the space $\hb$ contains a non-zero function $f(z) = \sum_{n \geq 0} f_nz^n$ which satisfies \begin{equation}
    \label{RapidDecayEq}
    \sup_{n \geq 0} \, \frac{|f_n|}{M_n} < \infty \end{equation}
    \item The function $b$ vanishes at some point $\lambda \in \D$, or there exists an arc $I \subset \T$ of positive length for which \[\int_I \log \Delta_b \, d\m > -\infty.\]
\end{enumerate}
\end{mainthm}

In $(iii)$, the condition of vanishing of $b$ at some $\lambda \in \D$ is the uninteresting case, sice then $\hb$ contains a rational function with no poles on $\cD$. For such a function $(ii)$ holds trivially.

To reach \thref{MainTheoremHbExistenceESD} we only really need the characterization of irreducibility of $\Po^2(\mu)$. Proof of the next theorem requires the full strength of the invariant subspace results developed in the first part of this article.

\begin{mainthm} \thlabel{MainTheoremHbDensityESD} Let $b: \D \to \D$ be an analytic function, and $b = BS_\nu U$ be the inner-outer factorization of $b$. The following three statements are equivalent.
\begin{enumerate}[(i)]
    \item The set of functions $f$ in $\hb$ which satisfy \eqref{ESDpropEq} is dense in $\hb$.
    \item For any admissible sequence $\{M_n\}_{n \geq 0}$, the set of functions $f$ in $\hb$ which satisfy \begin{equation}
    \sup_{n \geq 0} \, \frac{|f_n|}{M_n} < \infty \end{equation} is dense in $\hb$.
    \item The set $\co (\Delta_b)$ is a carrier for $\Delta_b$ and for the singular measure $\nu$.
\end{enumerate}
\end{mainthm}

\begin{example}
    For instance, by applying our theorem to the admissible sequence \eqref{ndivlognpSeq} for any $p > 1$, we get that the density in $\hb$ of functions $f(z) = \sum_{n \geq 0} f_nz^n$ satisfying
\[ \lim_{n \geq 0} \, |f_n|\exp(c n^\alpha) = 0\] simultaneously for any $c > 0$ and any $\alpha \in (0,1)$, is equivalent to condition $(iii)$ in \thref{MainTheoremHbDensityESD}. Roughly speaking, functions satisfying such decay a condition just barely fail to have an analytic continuation to a disk larger than $\D$.
\end{example}

\begin{example} \thlabel{HbEsetExample}
    Generalizing the setting of \thref{PointMassExample}, we may replace a point by a general closed subset $E$ of $\T$, and define the outer function $b_0: \D \to \D$ by specifying its modulus $|b_0(x)|$, $x \in \T$, to satisfy the equation \[ \sqrt{1-|b_0(x)|^2} = \Delta_{b_0}(x) := \frac{1}{2}\exp \Bigg(-\frac{1}{\dist{x}{E}}\Bigg)\] for $x \in \T \setminus E$, where $\dist{x}{E}$ is the Euclidean distance from the point $x$ to the closed set $E$, and $|b_0(x)| = 1$ for $x \in E$. We can easily check that \[\co (\Delta_{b_0}) = \T \setminus E.\] If $B$ is a Blaschke product and $S_\nu$ is a singuler inner function, then functions of rapid spectral decay will be dense in the space $\hb$, with $b := BS_\nu b_0$, if and only if $\nu(E) = 0$.
\end{example}

Our proof of \thref{MainTheoremHbDensityESD} depends crucially on \thref{CyclicityMainTheorem} and \thref{PermanenceMainTheorem}, but is otherwise similar to the proofs in \cite{limani_malman_2023} and \cite{limani2023problem}. However, in the present work we obtain new information on which functions in $\hb$ fail to be approximable by classes appearing in \thref{MainTheoremHbDensityESD}. These results are presented in Section \ref{DensityESDHbSec}. 

We mentioned earlier that our result is optimal. This is morally true, in the following sense. Assume that $M = \{M_n\}_{n \geq 0}$ is a logarithmically convex sequence which is not admissible according to \thref{AdmissibleSequenceDef}, because we have \begin{equation}
    \label{notAdmissibleSeqEq} \sum_{n \geq 0} \frac{\log M_n}{1+n^2} = -\infty.
\end{equation} For instance, $M$ could be defined by \eqref{ndivlognpSeq} for $p = 1$. If Volberg or Kriete and MacCluer were interested in approximations in $\hb$-spaces, they would have proved the following theorem by a use of their techniques in \cite{kriete1990mean} and \cite{vol1987summability}.

\begin{thm*}[\textbf{Volberg, Kriete-MacCluer}]
Let $M = \{M_n\}_{n \geq 0}$ be a logarithmically convex sequence satisfying the property \eqref{notAdmissibleSeqEq}. The following two statements are equivalent.

\begin{enumerate}[(i)]
    \item The space $\hb$ contains a non-zero function $f$ which satisfies \begin{equation*}
\sup_{n \geq 0} \, \frac{|f_n|}{M_n} < \infty. \end{equation*}
    \item The function $b$ vanishes at some point $\lambda \in \D$, or $b$ is non-extreme.  
\end{enumerate}  
\end{thm*}

We do not prove the above theorem in the present article. Its proof is completely analogous to the proof of \thref{MainTheoremHbExistenceESD}. The difference consists merely of a use of theorems and observations of Volberg and Kriete-MacCluer from the above mentioned papers, instead of main theorem of \cite{malman2023revisiting} as we do here in the proof of \thref{MainTheoremHbExistenceESD}. 

It follows that the investigation of existence and approximability properties in $\hb$ of functions with spectral decay satisfying at least \eqref{ESDpropEq} is essentially completed in \thref{MainTheoremHbExistenceESD} and \thref{MainTheoremHbDensityESD}.

\subsection{Additional comments, questions and conjectures}
\subsubsection{Work of McCarthy and Davis} The class of functions satisfying \eqref{ESDpropEq} has already appeared in the theory of de Branges-Rovnyak spaces. In \cite{davis1991multipliers}, McCarthy and Davis showed that a function $h$ satisfies \eqref{ESDpropEq} if and only if the multiplication operator $M_h$ acts boundedly on $\hb$ for \textit{all} non-extreme symbols $b$. In particular, this means that every space $\hb$ defined by a non-extreme symbol $b$ contains all functions satisfying \eqref{ESDpropEq}. Our \thref{MainTheoremHbExistenceESD} then establishes a converse statement: a characterization of $b$ for which $\hb$ contains no non-zero such functions. 

\subsubsection{Relation to Khrushchev's results} 
Khrushchev in \cite{khrushchev1978problem} studied a problem similar to one appearing in \thref{UncertThmRSD}. If $1_E$ is the characteristic function of a set $E$ contained in $\T$, and there exists a function $g$ living only on $E$ such that $\C_{g}$ in \eqref{CauchyIntUncertThm} has some regularity properties, then what can be said about $E$? Khrushchev used the phrase \textit{removal of singularities of Cauchy integrals} in the context of his study of nowhere dense $E \subset \T$ which support a function $g$ with a smooth Cauchy integral $\C_g$. Thus \textit{"removing"} the singularities of the irregular set $E$. His solution is given in terms of Beurling-Carleson sets. The weighted version of the problem replaces $1_E$ by a general weight $w$. 

In turn, \thref{MainTheoremHbDensityESD} can be seen as a solution to the problem of removal of singularities of normalized Cauchy integrals in context of the class \eqref{ESDpropEq}, where the possible existence of a singular part of the measure $\nu$ forces the normalization. Indeed, given a positive finite Borel measure $\nu$ on $\T$, we may ask if the space $L^2(\nu)$ contains a dense subset of functions $\tilde{f}$ for which the normalized Cauchy integral in \eqref{ClarkRepHb} (with $\nu_b$ replaced by $\nu$) satisfies \eqref{ESDpropEq}. The condition, in terms of the associated function $b = b_\nu$ given by \eqref{ClarkFormulaForb}, is given in $(iii)$ of \thref{MainTheoremHbDensityESD}. In this context, it would be of interest to characterize intrinsically the measures $\nu$ which correspond to $b$ satisfying condition $(iii)$ of \thref{MainTheoremHbDensityESD}.

\begin{question} Let $b: \D \to \D$ be an analytic functions which satisfies the condition $(iii)$ in \thref{MainTheoremHbDensityESD}. Can we describe the structure of the corresponding Aleksandrov-Clark measure $\nu_b$ of $b$ appearing in formula \eqref{ClarkFormulaForb} ?
\end{question}

\subsubsection{Logarithmic convexity of admissible sequences} In spite of some efforts, the author has not been able to remove the assumption of logarithmic convexity in \thref{AdmissibleSequenceDef}. Surely the most interesting admissible sequences, such as \eqref{ndivlognpSeq}, do satisfy such a conditon, but ideally one would like to remove this assumption. Logarithmic convexity of $\{M_n\}_{n \geq 0}$ plays its part in the proof of \thref{AdmissibleSequenceLemma}. In relation to that, we would like to answer the following question. 
\begin{question} If $c(x)$, $x > 0$, is an increasing, positive and continuous function which satisfies 
\begin{equation}
    \label{cIntEq} \int_1^\infty \frac{c(x)}{x^2} \, dx < \infty,
\end{equation} then under what additional conditons on $c$ may we replace $c(x)$ in \eqref{cIntEq} by its least concave majorant?
\end{question} 
Any interesting condition on $c$ which guarantees the above integrability property of its concave majorant will lead to slighly improved versions of our theorems.

\subsubsection{Non-integrability of $\log G$ as a sharp condition} \label{conjecturesSec}
Consider the condition 
\begin{equation} \label{LogIntTag}
\int_0^d \log (1/G(x)) \, dx  < \infty \tag{LogInt}\end{equation} for some $d > 0$. The condition \eqref{ExpDecTag} implies that our considered functions $G(x)$ will always fail to satisfy \eqref{LogIntTag}. In fact, it is (at least in the mind of the author) reasonable to conjecture that several of the results of this article should have sharp improvements in which the requirement for $G$ to satisfy \eqref{ExpDecTag} is replaced by the requirement for $G$ \textit{not} to satisfy \eqref{LogIntTag}. 
This condition is, in turn, equivalent to the statement that \[ \sum_{n \geq 0} \frac{\big( \log 1/M_n(G)\big)^2}{1+n^2} = \infty,\] where $\{M_n(G)\}$ is defined in \eqref{GMomentDef} and is the sequence of moments of the function $G$. This equivalence can be deduced using techniques appearing in Section \ref{MomentEstSection} below. The above condition appears in \cite{el2012cyclicity} as a necessary and sufficient condition for all singular inner functions to be cyclic in a space $\Po^2(\mu_\D)$ with $d\mu(z) = G(1-|z|)dA(z)$, and so $w = 0$ in contrast to the situation dealt with in the present article. 

For instance, a sharp version of the irreducibility of $\Po^2(\mu)$ with $\mu$ of the form \eqref{MuGeneralStructureEq} would follow if we could prove the following statement.

\begin{conjecture} \thlabel{conj1} Assume that $G$ fails to satisfy \eqref{LogIntTag} and $w \in \L^1(\T)$ is a non-negative weight on $\T$. If $\Po^2(\mu)$ of the form \eqref{MuGeneralStructureEq} is a space of analytic functions on $\D$, then the set $\co (w)$ of the weight $w \in \L^1(\T)$ is a carrier for $w$. 
\end{conjecture}

Given this result, one could attempt to combine our techniques appearing in Section \ref{CyclicSection} and those of El-Fallah, Kellay and Seip from \cite{el2012cyclicity} to prove the following strong version of both their result and our \thref{CyclicityMainTheorem}.

\begin{conjecture}
    In the setting of \thref{conj1}, a singular inner function $S_\nu$ is cyclic in the space of analytic functions $\Po^2(\mu)$ if and only if $\nu\big( \co (w) \big) = 0$.
\end{conjecture}

In relation to \thref{UncertThmRSD}, we expect the following improvement.
\begin{conjecture}
    The conclusion of \thref{UncertThmRSD} can be reached if \[\C_\nu(z) = \sum_{n \geq 0} \nu_n z^n\] merely satisfies \[ \sup_{n \geq 0} \,\frac{|\nu_n|}{M_n} < \infty\] for some (say, logarithmically convex) sequence $\{M_n\}_{n \geq 0}$ satisfying \[ \sum_{n \geq 0} \frac{\big( \log 1/M_n\big)^2}{1+n^2} = \infty.\]
\end{conjecture}

One can show, by considerations of examples, that all of the above conjectures imply sharp results.

\subsection{Outline of the rest of the article}

Section \ref{WizardHatSection} deals with construction of special domains which look like wizard hats and which support very large positive harmonic functions. We prove \thref{PermanenceMainTheorem} and \thref{CyclicityMainTheorem} in Sections \ref{PermanenceSection} and \ref{CyclicSection}, respectively. Proof of \thref{PermanenceMainTheorem} relies heavily on results of Section \ref{WizardHatSection}. The second part of the article starts in Section \ref{MomentEstSection}. There we deal with some preparatory estimates on moments sequences which are needed later. \thref{MainTheoremHbExistenceESD} and \thref{MainTheoremHbDensityESD} are proved in Sections \ref{ExistenceESDHbSec} and \ref{DensityESDHbSec}. The techniques used in these sections come from \cite{limani2023problem}, but we refine some of the methods and prove auxilliary results of hopefully independent interest. Finally, in Section \ref{ThmCProofSec}, we prove \thref{UncertThmRSD}. 

\subsection{Some notation}

For a measure $\mu$ on $\cD$ we will denote by $\mu_\D$ and $\mu_\T$ its restriction to $\D$ and $\T$, respectively. In some contexts we will also use the same notations $\mu_\D$ and $\mu_\T$ to emphasize that the considered measure lives only on $\D$ or $\T$. The area measure $dA$ will always be normalized by the condition $A(\D) = 1$, and a similar convention will be used also for the arc-length measure on the circle: $m(\T) = 1$. We let $\log^+(x) = \max \big( 0, \log x \big)$.

The symbol $\|\cdot \|_\mu$ always denotes the usual $\L^2(\mu)$-norm corresponding to the finite positive Borel measure $\mu$. For a set $E \subset \T$, we sometimes use the shorter notation $\L^2(E)$ to denote the space of functions on $\T$ which vanish outside of $E$ and are square-integrable with respect to the Lebesgue measure $\m$. The notation $\ip{\cdot}{\cdot}$ denotes different kinds of duality pairings between spaces. By $\ip{\cdot}{\cdot}_{\L^2}$ we will denote the standard inner product in $L^2(\T)$.

The operator $P_+: \L^2(\T) \to H^2$ is the orthogonal projection onto the Hardy space $H^2$. For a bounded analytic function $h$, the notation $T_{\conj{h}}: H^2 \to H^2$ stands for the co-analytic Toeplitz operator with symbol $h$, this operator being defined by the formula $T_{\conj{h}}f = P_+ \conj{h} f$.

\section{Wizard hats and their harmonic measures}
\label{WizardHatSection}

The proof of \thref{PermanenceMainTheorem} relies on a technique of restriction of a convergent sequence of analytic functions to a certain subdomain of $\D$. It is easier to construct the corresponding domain in the setting of a half-plane, and later use a conformal mapping argument. We will work in the upper half-plane $\H$. There, our domain looks like a wizard's hat (see Figure \ref{fig:wizardHat}).

Harmonic measures will play an important role in our discussion, so we start by recalling some basic related notions, and set some further notations. Let $\Omega$ be a Jordan domain in the plane. The domains which will appear in our context have a boundary consisting of a finite union of smooth curves. Let $\omega(z, E, \Omega)$ denote the harmonic measure of a segment $E$ of the boundary $\partial \Omega$, based at the point $z \in \Omega$.  Then \[ z \mapsto \omega(z, E, \Omega), \quad z \in \Omega \] is a positive harmonic function in $\Omega$ which extends continuously to the boundary $\partial \Omega$ except at the endpoints of $E$. It attains the boundary value $1$ at the relative interior of $E$, and boundary value $0$ on $\partial \Omega \setminus E$. Let $\mathcal{B}(\partial \Omega)$ denote the Borel $\sigma$-algebra on $\partial \Omega$. For each fixed $z_0 \in \Omega$, the mapping \[ A \mapsto \omega(z_0, A, \Omega), \quad A \in \mathcal{B}(\partial \Omega)\] defines a positive Borel probability measure on $\partial \Omega$. The reader can consult the excellent books by Garnett and Marshall \cite{garnett2005harmonic} and by Ransford \cite{ransford1995potential} for more background and other basic facts about harmonic measures which are used in this section.

Let \[\H = \{ z = x+iy \in \mathbb{C} : y > 0 \}\] denote the upper half-plane of $\mathbb{C}$. The main efforts of this section will go into estimation of the harmonic measure on a \textit{wizard hat domain} $W$. The domain is constructed from an interval $I \subset \R$ and a \textit{profile function} $p(x)$, $x \geq 0$, which by our definition is increasing, positive and continuous, smooth (say, continuously differentiable) for $x > 0$, and which satisfies $p(0) = 0$. Given a profile function $p$ and an interval $I = (a,b)$, we define the wizard hat $W$ to be the simply connected Jordan domain 
\begin{equation}
    \label{wizardHatEq} 
    W = W(p,I) := \Big\{ z = x+iy \in \H : x \in I, y < \min \big[ p(x-a), p(b-x) \big] \Big\}.
\end{equation} 
The boundary $\partial W$ is a piecewise smooth curve, with three smooth parts divided by three cusps. An example of a domain $W$, constructed from a profile function $p(x) = x^q$ for some $q > 1$, is marked by the shaded area in Figure \ref{fig:wizardHat}.
Our goal is to prove a result regarding existence of harmonic functions which grow rapidly along $\partial W \cap \H = \partial W \setminus \R$. 

\begin{defn}{\textbf{(Majorants)}}
\thlabel{RegMajorantDef} Let $d > 0$ be some positive number. A positive function $F: (0,d) \to [0, \infty)$ will be called a \textit{majorant} if it satisfies the following two properties:
\begin{enumerate}[(i)]
    \item $F(t)$ is a decreasing function of $t > 0$, and $\lim_{t \to 0^+} F(t) = +\infty$,
    \item $\int_0^d \log F(t) \, dt < \infty$.
\end{enumerate}
\end{defn}

The properties of $F$ appearing in \thref{RegMajorantDef} are related to growth estimates on functions in the investigated class of $\Po^2(\mu)$-spaces. See \thref{PointEvaluationBoundRegMajorant} below.

\begin{prop}
\thlabel{WizardHatMainProposition} Let $I \subset \R$ be a finite interval and $F$ be a majorant in the sense of \thref{RegMajorantDef}. There exists a profile function $p$, a wizard hat $W = W(p, I)$, and a positive harmonic function $u$ on $W$ which extends continuously to the boundary $\partial W$ except at the two cusps of $\partial W$ on $\R$, satisfies $u(x) = 0$ for $x$ in the interior of $I$, and $u(z) = F( \Im z)$ for $z \in \partial W \cap \H$.
\end{prop}

In order to prove \thref{WizardHatMainProposition}, we will need to estimate the harmonic measure $\omega(z_0, B_t, W)$ of the following piece $B_t$ of the boundary of $W$:

\begin{equation}
    \label{BtEq}
    B_t = \Big\{ z = x+iy\in \partial W : 0 < y, \, a < x < t \Big \}.
\end{equation}
See Figure \ref{fig:wizardHat}, where $B_t$ is marked. A result of Beurling and Ahlfors (see \cite[Theorem 6.1 of Chapter IV]{garnett2005harmonic}) can be applied to the union of $W$, $I$ and the reflected domain $\conj{W} = \{ \conj{z} : z \in W \}$ to obtain a good estimate for the harmonic measure of $B_t$.

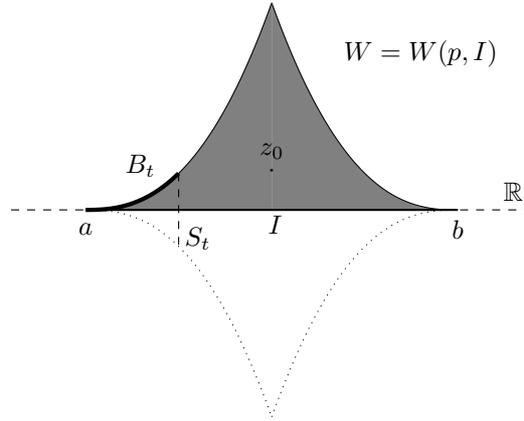
\begin{figure}
    \centering
    \begin{tikzpicture}
    \begin{axis}[
        ticks = none,
        axis x line=middle,
        axis y line=none,
        xmin = -0.2, xmax = 1.2,
        ymin = -1.1, ymax = 1.1,
        xlabel = {$\mathbb{R}$}, 
        xticklabels = \empty,
    ]
    \draw[name path= axis1, white, thick] (-0.2, 0) -- (0.5,0);
    \draw[name path= axis2, white, thick] (0.5, 0) -- (1.2,0);
    \draw[dashed] (-0.2, 0) -- (1.2,0);
    \draw[thick] (0,0) -- (1, 0) node[yshift =-0.2 cm, xshift=-2.4 cm]{$I$};

    \node [black] at (0.5, 0.3) {$z_0$};
    \node [black] at (0.5, 0.2) {$\cdot$};
    \node [black] at (0.9, 0.8) {$W = W(p, I)$};
    \node [black] at (0,-0.1) {$a$};
    \node [black] at (1,-0.1) {$b$};

    \draw[name path = h1, domain=0:0.5,smooth,variable=\t]
    plot ({\t},{6*\t^(2.5)});
    \draw[domain=0:0.25,smooth, ultra thick, variable=\t]
    plot ({\t},{6*\t^(2.5)}) node[xshift=-0.5cm, yshift=0.1cm]{$B_t$};

    \draw[name path = h2, domain=0.5:1,smooth,variable=\t]
    plot ({\t},{6*(1-\t)^(2.5)});

     \draw[domain=0:0.5,smooth,dotted, variable=\t]
    plot ({\t},{-6*\t^(2.5)});
    \draw[domain=0.5:1,smooth,, dotted, variable=\t]
    plot ({\t},{-6*(1-\t)^(2.5)});

    \addplot [
        thick,
        color=gray,
        fill=gray, 
        fill opacity=0.20
    ]
    fill between[
        of=h2 and axis2
    ];
    \addplot [
        thick,
        color=gray,
        fill=gray, 
        fill opacity=0.20
    ]
    fill between[
        of=h1 and axis1
    ];

    \draw[dashed] (0.25, 6*0.25^2.5) -- (0.25, -6*0.25^2.5) node[xshift=0.25cm, yshift=0.1cm]{$S_t$};

    \end{axis}    
\end{tikzpicture}
     \caption{The wizard hat $W$ and a piece $B_t$ of its boundary.}
    \label{fig:wizardHat}
\end{figure}

\begin{prop}\textbf{(Beurling-Ahlfors estimate)} Let $\theta$ be a positive continuous function defined on an interval $(a,b) \subset \R$, and let $\Omega$ be the domain \[ \Omega = \{ z = x+iy : |y| < \theta(x), \, a < x < b \Big\}. \] If $z_0 \in \Omega$ and $S_a = \{ z \in \partial \Omega : \Re z = a \}$ is the left vertical part of the boundary of $\Omega$, then \[ \omega(z_0, S_a, \Omega) \leq \frac{8}{\pi} \exp \Bigg( -2 \pi \int_a^{\Re z_0} \frac{dx}{\theta(x)}\Bigg). \]
\end{prop}
In Figure \ref{fig:wizardHat}, the symmetrized domain $\widetilde{W} := W \cup I \cup \conj{W}$ is bounded by the top part of the boundary of $W$ and the dotted reflection below the line $\R$. Let $\widetilde{W}_t$ be the domain obtained by cutting $\widetilde{W}$ along the cross-section $S_t = \{ z \in \widetilde{W} : \Re z = t \}$ and keeping the right part of the two resulting pieces. Define $W_t$ similarly (so that $W_t$ is the intersection of $\widetilde{W}_t$ and $\H$). The Beurling-Ahlfors estimate immediately implies that
\[ \omega(z_0, S_t, \widetilde{W}_t) \leq \frac{8}{\pi} \exp \Bigg( -2 \pi \int_t^{\Re z_0} \frac{1}{p(x-a)}\, dx\Bigg),\] where $z_0 \in W_t$ is as in Figure \ref{fig:wizardHat}. By a comparison of the values on $\partial W_t$ of the two harmonic functions $\omega(z, B_t, W)$ and $\omega(z, S_t, \widetilde{W}_t)$, and the maximum principle for harmonic functions, we get the inequality
\begin{equation}
    \omega(z, B_t, W) \leq \omega(z, S_t, \widetilde{W}_t), \quad z \in W_t.
\end{equation} In particular, this holds at $z_0$. We have obtained the following harmonic measure estimation.

\begin{prop} \thlabel{BtHarmEst} Let $W = W(p, I)$ be the wizard hat given by \eqref{wizardHatEq}, $B_t$ the piece of its boundary given by \eqref{BtEq} and $z_0 \in W$. Then \[ \omega(z_0, B_t, W) \leq \frac{8}{\pi} \exp \Bigg( -2 \pi \int_t^{\Re z_0} \frac{1}{p(x-a)}\, dx\Bigg)\] whenever $a < t < \Re z_0$.
\end{prop}

Given a majorant $F$ as in \thref{RegMajorantDef}, we will now show how to construct a profile function $p$ and harmonic function $u$ which satisfies the properties stated in \thref{WizardHatMainProposition}. Without loss of generality, we may assume that $I = (0,2)$. For some large integer $n_0 > 0$, let \begin{equation}
    \label{alphanDef}
    \alpha_n := 2^{-n-n_0}, \quad n \geq 1.
\end{equation} We define also the sequence \begin{equation}
    \label{gammaNdef}
    \gamma_n := \alpha_n \log F(\alpha_n), \quad n \geq 1
\end{equation} This sequence is positive if the integer $n_0$ in \eqref{alphanDef} is chosen large enough. Next, we make the following simple observation.

\begin{lem}
\thlabel{GammaSeqLemma}
For any $\epsilon > 0$, there exists an integer $n_0 > 0$ such that, with $\{\alpha_n\}_{n \geq 1}$ defined by \eqref{alphanDef} and $\{\gamma_n\}_{n \geq 1}$ defined by \eqref{gammaNdef}, we have \[ \sum_{n=1}^\infty \gamma_n < \epsilon.\]
\end{lem}

\begin{proof}
    Since $F(t)$ is a majorant, by part $(i)$ of \thref{RegMajorantDef} we have 
    \begin{align*}
    \int_0^{\alpha_1} \log F(t) \, dt & = \sum_{n=1}^\infty \int_{\alpha_{n+1}}^{\alpha_n} \log F(t) \, dt \\ 
    &\geq \sum_{n=1}^\infty \log F(\alpha_n) (\alpha_n - \alpha_{n+1}) \\ 
    & = \sum_{n=1}^\infty \log F(\alpha_n)\alpha_{n+1}\\
    & = \frac{1}{2} \sum_{n=1}^\infty \gamma_n,
    \end{align*} where we used that $\alpha_n - \alpha_{n+1} = \alpha_{n+1} = \alpha_n/2$. Now, by property $(ii)$ in \thref{RegMajorantDef} we have \[ \lim_{\alpha_1 \to 0^+} \int_0^{\alpha_1} \log F(t) \, dt = 0\] and so our claim follows.
\end{proof}

%We will now specify a decreasing sequence $\{t_n\}_{n \geq 1}$ by giving its initial value $t_1$ and its sequence of differences $\{ \Delta t_n\}_{n \geq 1}$, where $\Delta t_n :=  t_n - t_{n+1}$, $n \geq 1$. For some constants $A, B > 0$ and a choice of $t_1 < 1$, we let 
%\begin{equation}
%    \label{deltatNdef}
%    \Delta t_n = A \gamma_n + \frac{B}{n^2}, \quad n \geq 1.
%\end{equation}

We set $n_0$ to some value which ensures that 
\begin{equation}
    \label{GammaSumHalf}
    \sum_{n=1}^\infty \gamma_n < 1/2,
\end{equation} or in other words, the sum $\sum_{n=1}^\infty \gamma_n$ is less than one quarter of the length of the interval $I = (0,2)$. Further, we let $\{t_n\}_{n \geq 1}$ be a sequence of positive numbers starting with \[ t_1 = 1,\] which tends monotonically to $0$. We shall soon define $\{t_n\}_{n \geq 1}$ by specifying the sequence of differences $\{ \Delta t_n\}_{n \geq 1}$, where \[ \Delta t_n :=  t_n - t_{n+1}, \quad n \geq 1.\] The differences $\Delta t_n$ are positive numbers, and $t_2, t_3, \ldots$ will be recursively defined in terms of those differences by the relations \[t_2 = t_1 - \Delta t_1, \,  t_3 = t_2 - \Delta t_2, \] and so on. In order for so defined sequence $\{t_n\}_{n \geq 1}$ to converge to zero it is necessary and sufficient that
\begin{equation}
    \label{deltaTnReq}
    \sum_{n = 1}^\infty \Delta t_n = 1,
\end{equation} a requirement which we will later ensure. Given any $\{t_n\}_{n \geq 1}$ as above, a profile function $p$ may be readily constructed which satisfies
\begin{equation}
    \label{profileDef}
    p(t_n) = \alpha_n, \quad n \geq 1.
\end{equation}
Indeed, since the sequence $\{t_n\}_{n \geq 1}$ is assumed to be monotonically decreasing to zero, the function $p$ can be chosen to be smooth, increasing and positive for $t > 0$, and satisfy $p(0) = 0$. A proper choice of $\{t_n\}_{n \geq 1}$ will produce a wizard hat with our desired properties. Assume that $\{t_n\}_{n \geq 1}$ has been given, let $W = W(p, I)$ be the corresponding wizard hat, and $\omega(\cdot) = \omega(z_0, \cdot, \partial W)$ be the harmonic measure at some point $z_0 = 1 + y_0 i \in W$ which lies on the symmetry line of $W$. Let $\tilde{u}(z)$ be defined on $\partial W$ by \begin{equation}
\label{tildeudef}
\tilde{u}(z) = 
    \begin{cases}
        F(\Im z), &  z \in \partial W \cap \H,\\
        0, & z \in \partial W \cap \R. 
    \end{cases}
\end{equation}
We will ensure that $\tilde{u} \in \L^1(\omega)$. Since $F$ is decreasing, the definition of $W$ shows that for any $n \geq 1$ the values of the function $\tilde{u}$ on the arc $B_{t_n} \setminus B_{t_{n+1}}$ are dominated by its value at the point $z \in B_{t_n} \setminus B_{t_{n+1}}$ which lies closest to the real line $\R$, i.e., at the point $z = t_{n+1} + i p (t_{n+1})$. In other words, we have 
\begin{equation}
    \label{tildeuEstBt} \sup_{z \in B_{t_n} \setminus B_{t_{n+1}}} \tilde{u}(z) = F\big(p(t_{n+1})\big).
\end{equation} 
Moreover, from positivity and monotonicity of $p$, and from \thref{BtHarmEst}, we deduce the estimate 
\begin{align} \label{omegaBtnEst} \omega \big( B_{t_n} \setminus B_{t_{n+1}}\big) \leq \omega (B_{t_n})  & \leq \frac{8}{\pi} \exp \Bigg( - 2\pi \int_{t_n}^{t_{n-1}} \frac{1}{p(x)} dx \Bigg) \nonumber \\ & \leq \frac{8}{\pi} \exp \Bigg( - 2\pi \frac{\Delta t_{n-1}}{p(t_{n-1})} \Bigg)
\end{align} which holds for $n \geq 2$.

In order to have $\tilde{u} \in \L^1(\omega)$ it is sufficient to ensure that $\int_{B_{t_2}} \tilde{u}(z) d\omega(z) < \infty$. To this end, we use \eqref{tildeuEstBt} and \eqref{omegaBtnEst} to estimate  

\begin{align} \label{tildeUmainEst}
    \int_{B_{t_2}} \tilde{u}(z) d\omega(z) & = \sum_{n = 2}^\infty \int_{B_{t_n} \setminus B_{t_{n+1}}} \tilde{u}(z) d\omega(z) \nonumber \\
    &\leq \frac{8}{\pi}\sum_{n=2}^\infty F\big(p(t_{n+1})\big) \exp \Bigg( - 2\pi \frac{\Delta t_{n-1}}{p(t_{n-1})} \Bigg) \nonumber \\
    & = \frac{8}{\pi} \sum_{n=2}^\infty \exp \Bigg( \log F\big( p(t_{n+1}) \big) - 2\pi \frac{\Delta t_{n-1}}{p(t_{n-1})} \Bigg) \nonumber\\
    & = \frac{8}{\pi} \sum_{n=2}^\infty \exp \Bigg( \frac{1}{p(t_{n+1})}\Big( \log F\big( p(t_{n+1}) \big)p(t_{n+1}) - 2\pi \Delta t_{n-1} \frac{p(t_{n+1})}{p(t_{n-1})} \Big) \Bigg) \nonumber \\
    & = \frac{8}{\pi} \sum_{n=2}^\infty \exp \Bigg( 2^{n+1+n_0}\Big( \gamma_{n+1} - \frac{\pi \Delta t_{n-1} }{2} \Big) \Bigg).
\end{align}
In the last step we used \eqref{alphanDef}, \eqref{gammaNdef} and \eqref{profileDef}. We may now specify the values of $\{t_n\}_{n \geq 1}$ by setting the values of the differences: 
\begin{equation}
    \label{DeltaTdef}
    \Delta t_{n-1} = \frac{A}{n^2} + \frac{2 }{\pi }\gamma_{n+1}, \quad n \geq 2.
\end{equation} for an appropriate constant $A > 0$ which ensures the necessary summation condition \eqref{deltaTnReq}. This can be done, since \[\sum_{n=2} \frac{2 }{\pi}\gamma_{n+1} < 1/\pi < 1\] by \eqref{GammaSumHalf}. We obtain from \eqref{tildeUmainEst} that \[ \int_{B_{t_2}} \tilde{u}(z) d\omega(z) \leq \frac{8}{\pi} \sum_{n=2}^\infty \exp \Big( - \frac{A \pi}{2}\frac{ 2^{n+1+n_0}}{n^2} \Big) < \infty.\] Consequently, with this definition of $\{t_n\}_{n \geq 1}$ and a corresponding profile function $p$, we have that $\tilde{u} \in \L^1(\omega)$.

We may now complete the proof of \thref{WizardHatMainProposition}. Let $W$ and $p$ be chosen as above, and $\phi: \D \to W$ be a conformal mapping which maps $0 \in \D$ to $z_0 \in W$. Since $W$ is a Jordan domain, $\phi$ extends to a homeomorphism between $\D \cup \T$ and $W \cup \partial W$. If $v: \T \to [0,\infty)$ is defined by $v = \tilde{u} \circ \phi$, then a change of variables shows that \[ \int_\T v \, d\m = \int_{\partial W} \tilde{u} \, d\omega < \infty.\] Since $\tilde{u}$ is continuous on all of $\partial W$ except at the two cusps of $\partial W$ on $\R$, the function $v$ is continuous on $\T$ except at the two points which map under $\phi$ to the cusps. We verified above that $v \in \L^1(\T)$, so we may extend $v$ to $\D$ by means of its Poisson integral. This extension is continuous in $\D \cup \T$ except at the two points corresponding to the cusps. If we define $u$ in $W$ by $u = v \circ \phi^{-1}$, then $u$ is the harmonic function sought in \thref{WizardHatMainProposition}.

\section{Proper invariant subspaces generated by singular inner functions}
\label{PermanenceSection}

The goal of this section is to prove \thref{PermanenceMainTheorem}.

\subsection{Technical lemmas}

Similarly to Section \ref{WizardHatSection}, we prove the next lemma in the upper half-plane $\H = \{ z \in \mathbb{C} : \Im z > 0 \}$. This is done, again, only for convenience. An elementary conformal mapping argument will carry the result over to the intended domain $\D$. 

In this section, the Lebesgue measure (length measure) on $\R$ will be denoted by $dt$, and the $dt$-measure of a set $A$ will be denoted by $|A|$, similar to lengths of sets on the circle $\T$ (this should not cause confusion). The algebra of bounded analytic functions in $\H$ will be denoted by $H^\infty(\H)$. In the proofs below we shall use some basic facts regarding $H^\infty(\H)$, and in particular some factorization results. An exposition of the relevant background can be found in \cite[Chapter 11]{duren1970theory}, \cite[Chapter II]{garnett} or \cite[Chapter VI]{koosis}. 

Every function $h \in H^\infty(\H)$ admits an \textit{inner-outer factorization} into \begin{equation}
    \label{InnerOuterFactHalfplane}
    h(z) = c e^{iaz} B(z) S_\nu(z) U(z), \quad z \in \H.
\end{equation} Here $c$ is some unimodular constant, $a \geq 0$, $B$ is a Blaschke product given by 
\[ B(z) = \Big(\frac{z-i}{z+i}\Big)^m \prod_{\substack{h(\alpha) = 0, \\ \alpha \neq i}} \frac{i - \conj{\alpha}}{i - \alpha}\cdot \frac{z - \alpha}{z - \overline{\alpha}}, \quad z \in \H\] where $m$ is a non-negative integer, $S_{\nu_h}$ is a singular inner function given by \[ S_{\nu_h}(z) = \exp \Bigg( -\frac{1}{i \pi}\int_{\R} \frac{(1+tz)}{(t-z)} \frac{d\nu_h(t)}{(1+t^2)} \Bigg), \quad z \in \H\] where $\nu_h$ is a singular positive Borel measure on $\R$, and $U$ is the outer function given by \[ U(z) = \exp \Bigg( \frac{1}{i \pi} \int_\R \frac{(1 + tz)\log |h(t)|}{(t-z)(1+t^2)} \, dt \Bigg), \quad z \in \H.\]
The measures $(1+t^2)^{-1}d\nu_h(t)$ and $(1+t^2)^{-1}\log|h(t)|\,dt$ appearing in the integrals above are both finite. It follows from this factorization that we have  \begin{align} \label{loghFormula2}
        \log |h(z)| & = - \alpha y + \log |B(z)| \\
        &- \frac{1}{\pi} \int_\R \frac{y}{(x-t)^2 + y^2} d\nu_h(t) \nonumber \\
        & + \frac{1}{\pi} \int_\R \frac{y}{(x-t)^2 + y^2} \log|h(t)|dt, \quad z = x+iy \in \H \nonumber 
\end{align} 
The last two terms in \eqref{loghFormula2} represent the Poisson integrals $\Po_{\nu_h}$ and $\Po_{\log|h|}$ of the measure $\nu_h$ and of the function $\log |h|$, respectively. 

\begin{lem}
    \thlabel{IntervalWeakStarConvLemma}
    Let $J$ be a finite open interval of $\R$. With notation as above, as $y \to 0^+$, the restrictions to $J$ of the measures $\log |h(t+iy)| dt$ converge weak-star to the restriction to $J$ of the measure $\log|h(t)| dt - d\nu_h(t)$.
\end{lem}

The lemma follows easily from results presented in de Branges' book \cite[Theorem 3 and Problem 26]{de1968hilbert}. We sketch an argument for the reader's convenience.

\begin{proof}[Proof of \thref{IntervalWeakStarConvLemma}] If $\phi$ is any smooth function which is supported on a compact subset of $J$, then by the symmetry of the Poisson kernel, we have \begin{equation} \label{poissonEqIntSym}
    \int_\R \phi(t) \Po_{\nu_h}(t+iy) dt = \int_\R \Po_\phi(t+iy) d\nu_h(t),
\end{equation} where $\Po_\phi$ is the Poisson integral of $\phi$. Since $\phi$ is uniformly continuous on $\R$, $\Po_\phi(t+iy) \to \phi(t)$ uniformly in $t$ as $y \to 0^+$. Moreover, the compact support of $\phi$ implies that $|\Po_\phi(t+iy)| = \mathcal{O}(1/t^2)$ as $|t| \to \infty$, uniformly in, say, $y \in (0,1)$. Thus expression \eqref{poissonEqIntSym} and the finiteness of the measure $(1+t^2)^{-1}d\nu_h(t)$ implies now that \[ \lim_{y \to 0^+} \int_\R \phi(t) \Po_{\nu_h}(t+iy) dt = \int_\R \phi(t) d\nu_h(t),\] and we have shown that $\Po_{\nu_h}(t+iy)dt$ converges weak-star to $\nu_h$ on the interval $J$. By the same argument $\Po_{\log|h|}(t+iy)dt$ converges weak-star to $\log|h(t)| dt$ on $J$. 

We consider now the measures $\log |B(t+iy)| dt$. Jensen's inequality for the upper-half plane (see, for instance, \cite[p. 35]{havinbook}) implies that \[ \log |B(i+iy)| \leq \int_\R \frac{\log |B(t+iy)|}{\pi(1+t^2)} dt.\] Thus by letting $y$ tend to $0^+$, we obtain \[ \log |B(i)| \leq \liminf_{y \to 0^+} \int_\R \frac{\log |B(t+iy)|}{\pi(1+t^2)} dt \leq 0.\] The last inequality is trivial, since $\log |B|$ is negative in $\H$. For a finite Blaschke product $B_0$, the limit between the inequalities above is certainly equal to $0$. Thus we obtain \[ \log |B(i)| - \log |B_0(i)| \leq \liminf_{y \to 0^+} \int_\R \frac{\log |B(t+iy)|}{\pi(1+t^2)} dt \leq 0.\] Now let $B_0$ tend to $B$ through a sequence of finite partial products of $B$ to obtain \[ \lim_{y \to 0^+} \int_\R \frac{\log |B(t+iy)|}{\pi(1+t^2)} dt = 0.\] This says that the restriction to $J$ of $\log |B(t+iy)|dt$ converges to $0$ even in variation norm. 

The expression \eqref{loghFormula2} now implies the weak-star convergence result we are seeking.
    
\end{proof}

\begin{defn}{\textbf{(Uniform absolute continuity)}} \thlabel{uniAbsContDef}
    If $\{f_n \, dt\}_{n \geq 1}$ is a sequence of non-negative absolutely continuous Borel measures on $\R$ and $I \subset \R$ is an interval, then we will say that the sequence $\{f_n \, dt \}_{n \geq 1}$ is \textit{uniformly absolutely continuous on} $I$ if to each $\epsilon > 0$ there corresponds a $\delta > 0$, independent of $n$, such that for Borel sets $A$ we have \[ A \subset I, \, |A| < \delta \implies \int_A f_n \,dt < \epsilon.\]
\end{defn}

Recall that the notion of a majorant has been introduced in \thref{RegMajorantDef}.

\begin{lem} \thlabel{permanenceLemma} Let $I$ be a finite interval of the real line $\R$, $\theta = S_\nu$ a singular inner function in $\H$ defined by a singular measure $\nu$ supported in the interior of $I$, and $\{h_n\}_{n \geq 1}$ a sequence of functions in $H^\infty(\H)$ such that \[ \lim_{n \to \infty}  \, \theta(z) h_n(z) = h(z), \quad z \in \H,\] where $h \in H^\infty(\H)$ is a non-zero function. Assume that \begin{enumerate}[(i)]
    \item there exists a majorant $F$ for which we have \[ \sup_{z = x+iy \in R} \, |\theta(z)h_n(z)|\exp\big(- F(y) \big) < C \] for some constant $C > 0$ independent of $n$, and where $R$ is some rectangle in $\H$ with base $I$: \[ R = R(I,d) := \{ z = x+iy \in \H : x \in I, y < d \},\]
    \item the sequence of positive Borel measures $\{\log^+|h_n|dt\}_{n \geq 1}$ is uniformly absolutely continuous on an interval larger than $I$.
\end{enumerate}
Then $h/\theta \in H^\infty(\H)$.
\end{lem}

\begin{proof}
The assumption $(ii)$ implies that \[ \sup_{n} \int_I \log^+ |h_n| \, dt < \infty. \] So, denoting by $1_I$ the characteristic function of the interval $I$ and by passing to a subsequence, we can assume that the measures $1_I \log^+|h_n|dt$ converge weak-star to a non-negative measure $\nu_0$ supported on $I$. The measure $\nu_0$ must be absolutely continuous with respect to $dt$: any set $N \subset I$ of $dt$-measure zero can be covered by an open set $U$ of total length arbitrarily small, and then we can use $(ii)$ to conclude that \[ \nu_0(N) \leq \nu_0(U) \leq \liminf_{n \to \infty} \int_U \log^+|h_n| dt < \epsilon\] for any $ \epsilon > 0$. Consequently $d\nu_0 = w\, dt$ for some non-negative $w \in \L^1(I)$. We denote by $u_I$ the harmonic function in $\H$ which is the Poisson extension of the measure $d\nu_0 = w\,dt$ to $\H$: \[u_I(z) = \frac{1}{\pi}\int_I \frac{y}{(x-t)^2 + y^2} w(t) dt, \quad z = x+iy \in \H. \] Let also $u_n$ denote the Poisson extension of the measure $1_I \log^+|h_n|dt$: \[ u_n(z) = \frac{1}{\pi} \int_I \frac{y}{(x-t)^2 + y^2} \log^+|h_n(t)| \, dt, \quad z = x+iy \in \H.\]
The assumption $(i)$ implies that \[ \log|\theta(z)| + \log |h_n(z)| \leq c + F(y), \quad z = x+iy \in R\] for some positive constant $c > 0$. By \thref{WizardHatMainProposition}, there exists a wizard hat domain $W = W(p,I)$ and a corresponding positive harmonic function $u$ defined on $W$ which satisfies \[ \log|\theta(z)| + \log |h_n(z)| \leq u(z), \quad z \in \partial W \cap R.\] By the assumption that the singular measure $\nu$ defining $\theta$ is supported in the interior of $I$, it follows that $\theta$ is analytic and non-zero in a neighbourhood of $\partial W \cap R$, and so $\log |\theta(z)|$ is bounded on $\partial W \cap R$. Therefore, by possibly replacing $u$ by a positive scalar multiple of itself, in fact we have that \begin{equation}
    \label{h_nuIneqTopPart} \log |h_n(z)| \leq u(z), \quad z \in \partial W \cap R.
\end{equation}
For the bottom side $I$ of the wizard hat, we have the non-tangential boundary value inequality 
\begin{equation} \label{h_nuIneqBottomPart}
\log |h_n(x)| \leq u_n(x)    
\end{equation} for $dt$-almost every $x \in I$. This follows immediately from elementary boundary behaviour properties of Poisson integrals. We would like to conclude from the two inequalities \eqref{h_nuIneqTopPart} and \eqref{h_nuIneqBottomPart} that 
\begin{equation}
\label{AllOfTIneq}
\log |h_n(z)| \leq u(z) + u_n(z), \quad z \in W.
\end{equation} Indeed such a generalization of the maximum principle holds, and we will carefully verify this claim in \thref{TdomLemma} below. Assuming the claim, we recall that \[h_n(z) \to h(z)/\theta(z), \quad n \to +\infty\] in all of $\H$, and so by letting $n \to +\infty$ we obtain, from \eqref{AllOfTIneq} and the earlier mentioned weak-star convergence of measures (which guarantees that $u_n(z) \to u_I(z)$ for $z \in \H$), that 
\begin{equation}
        \label{FundamentalTIneq}
        \log |h(z)| - \log |\theta(z)| \leq u(z) + u_I(z), \quad z \in W.
\end{equation}
Let $J$ be some interval containing the support of $\nu$, and which is strictly contained in $I$. By \thref{IntervalWeakStarConvLemma}, as $y \to 0^+$, the restrictions to $J$ of the real-valued measures $\log |h(t + iy)| dt$ converge weak-star to the restriction to $J$ of the measure $\log |h(t)| dt - d \nu_h(t)$. Similar claims hold for the Poisson integrals \begin{equation}
\label{logThetaFormula}
-\Po_\nu(z) = \log |\theta(z)| = -\frac{1}{\pi} \int_\R \frac{y}{(x-t)^2 + y^2} d\nu(t), \quad z = x+iy \in \H
\end{equation} and $\Po_w = u_I$. For sufficiently small $y > 0$ we have that $J + iy := \{ x + iy : x \in J \} \subset W$. Thus from the weak-star convergence of measures discussed above, the inequality \eqref{FundamentalTIneq}, and the fact that $u \equiv 0$ on $J$, we obtain the real-valued measure inequality \[ \log |h(t)| dt - d\nu_h(t) + d\nu(t) \leq w(t) dt \quad \text{ on } J.\] This measure inequality is to be interpreted in the following way: \[w(t) dt - \log |h(t)| dt + d\nu_h(t) - d\nu(t)\] is a non-negative measure on $J$. The $dt$-singular part of this measure is $d\nu_h - d\nu$, which is thus non-negative on $J$. Since $d\nu$ is supported inside $J$, in fact $d\nu_h - d\nu$ is non-negative in all of $\R$. Now subtracting \eqref{logThetaFormula} from \eqref{loghFormula2}, using the inequality $-\alpha y + \log |B(z)| \leq 0$ and the non-negativity of $d\nu_h - d\nu$ and of the Poisson kernel, we get for $ z = x+iy \in \H$ that \begin{align*} \log |h(z)| - \log |\theta(z)| \leq  & \frac{1}{\pi} \int_\R \frac{y}{(x-t)^2 + y^2} \log|h(t)|dt \\ \leq  & \frac{1}{\pi} \int_\R \frac{y}{(x-t)^2 + y^2} \log (\|h\|_\infty) dt \\ = & \log (\|h\|_\infty).\end{align*}
By exponentiating, we finally obtain \[ |h(z)/\theta(z)| \leq \|h\|_{H^\infty(\H)}, \quad z \in \H.\]       
\end{proof}

We need to verify the claim made in the course of the proof of \thref{permanenceLemma} which lead to the fundamental inequality \eqref{AllOfTIneq}.

\begin{lem} \thlabel{TdomLemma} Let the wizard hat $W = W(p,I)$ be as in the proof of \thref{permanenceLemma}, $f$ be a bounded analytic function in $W$ and $u$ be a positive harmonic function in $W$. Assume that both $f$ and $u$ extend continuously to $\partial W \cap \H$ and also that both have non-tangential limits almost everywhere on $I$. If we have that $\log |f(z)| \leq u(z)$ for $z \in \partial W \cap \H$, and moreover that the non-tangential limits of $\log |f|$ on $I$ are $dt$-almost everywhere dominated by the non-tangential limits of $u$ on $I$, then $\log |f(z)| \leq u(z)$ for all $z \in W$.
\end{lem}

\begin{proof} Let $\phi: \D \to W$ be a conformal mapping. The local smoothness of the boundary of $W$ and basic conformal mapping theory ensure that $\phi$ is conformal at almost every point of $\T$ (see \cite[Chapter V.5]{garnett2005harmonic}). This implies that the functions $\log |f \circ \phi|$ and $u \circ \phi$, which are defined in $\D$, have non-tangential limits almost everywhere on $\T$. Let $\widetilde{u \circ \phi}$ be a harmonic conjugate of $u \circ \phi$ in $\D$ and consider the function $H(z) = \exp\Big(- u \circ \phi(z) - i\widetilde{u \circ \phi}(z) \Big)$, $z \in \D$. By positivity of $u$, the function $H$ is bounded in $\D$, and our assumptions leads to the conclusion that the non-tangential boundary values on $\T$ of the bounded function $(f \circ \phi)(z)H(z) \in H^\infty(\D)$ are not larger than 1 in modulus. Thus by basic function theory in $\D$, we obtain the inequality $|(f \circ \phi)(z)H(z)| \leq 1$ for all $z \in \D$. This easily translates into $\log |f(z)| \leq u(z)$ for $z \in W$.
\end{proof}

We will need \thref{permanenceLemma} in the disk $\D$. Here is the precise statement which we will use. The uniform absolute continuity of sequences of Borel measures on arcs $I$ of the circle $\T$ is defined analogously to how it was defined in \thref{uniAbsContDef} for intervals on the line.

\begin{cor}
\thlabel{PermanenceLemmaDisk} Let $I$ be an arc properly contained in the circle $\T$, $\theta = S_\nu$ be a singular inner function in $\D$ defined by a singular measure $\nu$ supported in the interior of $I$, and $\{h_n\}_{n \geq 1}$ be a sequence of functions in $H^\infty$ such that \[ \lim_{n \to +\infty}  \, \theta(z) h_n(z) = h(z), \quad z \in \D\] where $h \in H^\infty$. Assume that \begin{enumerate}[(i)]
    \item there exists a majorant $F$ for which we have \[ \sup_{z \in \D} \, |\theta(z)h_n(z)|\exp\Big(-F\big(1-|z|\big) \Big) < C \] for some positive constant $C > 0$ independent of $n$,
    \item the sequence of positive Borel measures $\{\log^+|h_n|d\m\}_{n \geq 1}$ is uniformly absolutely continuous on an arc larger than $I$.
\end{enumerate}
Then $h/\theta \in H^\infty$.
\end{cor}

It is easy to see that \thref{permanenceLemma} implies \thref{PermanenceLemmaDisk}. Indeed, if $\phi: \H \to \D$ is a conformal map for which $\phi^{-1}(I)$ is a finite segment on $\R$, then the distortion of lengths and distances by the map $\phi$ is bounded above and below near $\phi^{-1}(I)$ and $I$, since $\phi$ is a bi-Lipschitz bijection between some open sets containing $\phi^{-1}(I)$ and $I$. For instance, the growth condition $(i)$ in \thref{PermanenceLemmaDisk} for the sequence $\{S_\nu h_n\}_{n \geq 1}$ is easily translated into a corresponding condition $(i)$ in \thref{permanenceLemma} for the sequence $\{S_{\tilde{\nu}} \tilde{h}_n\}_{n \geq 1}$, where $S_{\tilde{\nu}} := S_\nu \circ \phi$ and $\tilde{h}_n := h_n \circ \phi$, by replacing $F(t)$ with a new majorant of the form $\tilde{F}(t) := F(at)$ for some $a > 0$. Moreover, the mapping $\phi$ will preserve the uniform absolute continuity properties of the corresponding measures. Thus \thref{PermanenceLemmaDisk} can readily be deduced from \thref{permanenceLemma} and a change of variables argument.

\subsection{Proof of \thref{PermanenceMainTheorem}}

\thref{PermanenceMainTheorem} follows almost immediately from \thref{PermanenceLemmaDisk}, we just need to verify that a bounded sequence in the corresponding $\Po^2(\mu)$-space satisfies properties $(i)$ and $(ii)$ in \thref{PermanenceLemmaDisk}. This is done in the next two lemmas. 

\begin{lem} \thlabel{PointEvaluationBoundRegMajorant} Let \[ d\mu_\D(z) = G(1-|z|) dA(z),\] where $G$ satisfies the condition \eqref{LogLogIntTag} appearing in Section \ref{introsec}. For $z \in \D$, denote by \[ E_z := \sup_{ \substack{ f \in \mathcal{P}, \\\|f\|_{\mu_\D} = 1 }} |f(z)| \] the norm of the evaluation functional $z \mapsto f(z)$ on $\Po^2(\mu_\D)$. There exists a majorant $F$ such that \[ E_z \leq \exp \big( F(1-|z|) \big), \quad z \in \D.\]
\end{lem}

\begin{proof} Fix $z \in \D$, $\delta = (1-|z|)/2$ and let $B(z,\delta)$ denote the ball around $z$ of radius $\delta$. By subharmonicity of the function $z \mapsto |f(z)|$ and the Cauchy-Schwarz inequality, we have 
\begin{align*}
    |f(z)|  &\leq \frac{1}{\delta^2} \int_{B(z,\delta)} |f(z)| dA(z) \\ &\leq \frac{1}{\delta^2} \|f\|_{\mu_\D} \sqrt{\int_{B(z,\delta)} \frac{1}{G(1-|z|)} dA(z)}.
\end{align*}
Since $G$ is assumed to be an increasing function, we may estimate the integral inside the square root by 
\[ \int_{B(z,\delta)} \frac{1}{G(1-|z|)} dA(z) \leq \frac{1}{\delta^2} \frac{1}{G\big((1-|z|)/2\big)}.\] Putting this into the previous estimate, we obtain \[ |f(z)| \leq \frac{\|f\|_{\mu_\D}}{\delta ^3 } \sqrt{\frac{1}{G\big((1-|z|)/2 \big)}} = \frac{8\|f\|_{\mu_\D}}{(1-|z|)^3 } \sqrt{\frac{1}{G\big((1-|z|)/2 \big)}}\]
Now set \[ F(t) := \log \Big( \frac{8}{t^3 } \Big)  + \frac{1}{2}\log \Big( \frac{1}{G(t/2)} \Big), \quad t \in (0,1].\] By the above estimate, the norm $E_z$ of the evaluation functional is bounded by \[ E_z \leq \exp \big ( F(1-|z|) \big), \quad z \in \D.\] Moreover, $F$ is a decreasing function, and by virtue of $G$ satisfying \eqref{LogLogIntTag}, $F$ also certainly satisfies \[ \int_0^d \log F(t)\, dt < \infty\] if $d > 0$ is some small number. Thus $F$ is a majorant in the sense of \thref{RegMajorantDef}.
\end{proof}

\begin{lem} \thlabel{UniAbsContLemma} Assume that the weight $w$ satisfies \[ \int_I \log w \, d\m > -\infty\] for some arc $I \subset \T$. If $\{f_n\}_{n \geq 1}$ are positive functions such that \[ \int_I f_n^p w \, d\m < C, \quad n \geq 1\] for some constant $C > 0$ and some $p > 0$, then the sequence $\{\log^+ f_n \, d\m \}_{n \geq 1}$ is uniformly absolutely continuous on $I$.
\end{lem}

\begin{proof}
    Note that     
    \begin{align*}
        \log^+ f_n & \leq \log^+( f_n w^{1/p}) + \log^+(w^{-1/p}) \\
                & \leq \frac{1}{p}\log^+( f_n^p w) + \frac{1}{p} \log^+ (1/w) \\ & := g_n + g,
    \end{align*} where it follows from the assumption that $g_n$ are positive functions which form a bounded subset of (say) $\L^2(I)$, and $g \in \L^1(I)$. Clearly, if $A$ is a Borel subset of $I$, then by Cauchy-Schwarz inequality we obtain \[ \int_A g_n \, d\m \leq \sqrt{|A}| \cdot \|g_n\|_{\L^2(I)},\] so that the family $\{g_n d\m\}_{n \geq 1}$ is uniformly absolutely continuous on $I$. Then the above inequalities imply that $\{\log^+ f_n \, d\m \}_{n \geq 1}$ is a uniformly absolutely continuous sequence on $I$.
\end{proof}

\begin{proof}[Proof of \thref{PermanenceMainTheorem}]
Let $h \in [S_\nu] \cap \N^+$. Since $h \in [S_\nu]$, there exists a sequence of polynomials $\{p_n\}_{n \geq 1}$ such that $S_\nu  p_n$ converges to $h$ in the norm of $\Po^2(\mu)$. Multiplying $h$ by a suitable bounded outer function $u$ we can ensure that $hu \in H^\infty$, and that $S_\nu  p_n  u$ converges to $hu$ (see the discussion following \thref{ProdOfCyclicLemma} below). Let $\{K_j\}_{j}$ be an increasing sequence of compact sets which are finite unions of intervals and such that $\cup_{j} K_j = \co (w)$. By \thref{PermanenceLemmaDisk}, \thref{PointEvaluationBoundRegMajorant} and \thref{UniAbsContLemma}, whenever $\nu_j$ is the restriction of $\nu$ to the compact subset $K_j$, we have that $hu/S_{\nu_j} \in H^\infty$ with the bound $\|hu/S_{\nu_j}\|_\infty \leq \|hu\|_\infty$. The assumption that $\nu(\T) = \nu\big( \co (w) \big)$ means that the restrictions $\nu_j$ converge weak-star to the measure $\nu$. Thus \[ |h(z)u(z)/S_\nu(z)| = \lim_{j \to \infty} |h(z)u(z)/S_{\nu_j}(z)| \leq \|hu\|_\infty, \quad z \in \D \] In particular, since $u$ is outer, it follows that $S_\nu$ divides the inner factor of $h$. Thus $h/S_\nu \in \N^+$.
\end{proof}

\section{Cyclic singular inner functions}
\label{CyclicSection}

In this section we will study the cyclicity of singular inner functions, and prove \thref{CyclicityMainTheorem}.

\subsection{Weak-star approximation of singular measures, with obstacles} 

The cyclicity in $\Po^2(\mu)$ of the singular inner function $S_\nu$ will follow from the existence of a sequence of non-negative bounded functions $\{f_n\}_{n \geq 1}$ for which the measures $\{f_n \, d\m\}_{n \geq 1}$ converge weak-star to $\nu$. In our context the functions $f_n$ will have to satisfy a severe restriction on their size, namely \begin{equation}
    \label{ObstacleForfn}
    0 \leq f_n(x) \leq \log^+(1/w(x)), \quad x \in \T.
\end{equation} In case $w(x) = 0$, the right-hand side is to be interpreted as $+\infty$ (i.e., no size restriction on $f_n$ at the point $x$). Essentially, the obstacle \eqref{ObstacleForfn} prohibits the existence of an approximating sequence $\{f_n d\m\}_{n \geq 1}$ if some part of the mass of $\nu$ is located in "wrong" places on $\T$. However, if $\nu$ is carried outside of the core set of $w$, then such a sequence exists. This is the content of the next lemma.

\begin{lem} \thlabel{WeakStarSeqObstacleLemma}
    Let $\nu$ be a positive singular Borel measure on $\T$ which satisfies \[\nu\big( \co (w)\big) = 0.\] Then there exists a sequence of non-negative bounded functions $\{f_n\}_{n \geq 1}$ satisfying the following three properties:    
    \begin{enumerate}[(i)]
        \item the non-negative measures $\{f_n \, d\m\}_{n \geq 1}$ converge weak-star to $\nu$,
        \item $\int_\T f_n \, d\m = \nu(\T)$,
        \item the functions $f_n$ obey the bound \eqref{ObstacleForfn}.
    \end{enumerate}
\end{lem}

\begin{proof}
    Let us first suppose that $\nu$ assigns no mass to any singletons, so that $\nu( \{ x\}) = 0$ whenever $x \in \T$. For any positive integer $n$, we let $D_n$ be the family of $2^n$ disjoint open dyadic intervals, each of length $2\pi \cdot 2^{-n} $, such that their union covers the circle $\T$ up to finitely many points, and such that the system $\cup_{n \geq 1} D_n$ possesses the usual dyadic nesting property: each $d \in D_n$ is contained in a unique $d' \in D_{n-1}$. Fixing an integer $n \geq 1$, we will specify how to define $f_n$ on each of the intervals $d_j \in D_n$, $1 \leq j \leq 2^n$, in such a way that the above three properties hold. 

    If $\nu(d_j) = 0$, then we simply set $f_n \equiv 0$ on $d_j$. Conversely, if $\nu(d_j) > 0$, then since $\nu\big( \co (w) \big) = 0$, it must be that $\nu(d_j) = \nu(d_j \setminus \co (w) ) > 0$. It follows that $d_j \setminus \co (w) $ is non-empty. Pick some point $x \in d_j \setminus \co (w) $. For any open interval $I$ which contains $x$ in its interior we have $\int_I \log^+(1/w) \,d\m = +\infty$ (else $x$ would have been a member of $\co (w)$). Pick such an interval $I$ which is contained within $d_j$. If there exists a subset $A \subset I$ satisfying $m(A) = |A| > 0$ on which $w$ is identically zero, then we may set \[ f_n(x) = \nu(d_j)|A|^{-1}1_A(x), \quad x \in d_j\] where $1_A$ is the characteristic function of $A$. In case that such a set does not exist, then $w > 0$ almost everywhere on $I$, and we must have \[ +\infty = \int_I \log^+(1/w)\, d\m = \lim_{c \to 0^+} \int_{I \cap \{ w > c\}} \log^+(1/w)\, d\m \] so that \[ \nu(d_j) < \int_{I \cap \{ w > c\}} \log^+(1/w)\, d\m < +\infty \] for some small $c > 0$. By absolute continuity of the finite measure \[ \log^+(1/w) 1_{I \cap \{ w > c\}}d\m \] there must then exist a set $B \subset I \cap \{ w > c\}$ for which we have precisely \[ \nu(d_j) = \int_B \log^+(1/w)\, d\m\] We pick such $B$ and define \[ f_n(x) = \log^+(1/w(x)) 1_B(x), \quad x \in d_j.\]
    Note that $f_n(x) \leq \log^+(1/c)$ on $d_j$. For definiteness, we can set $f_n$ to be equal to $0$ on the finitely many points outside of $\cup_{j=1}^{2^n} d_j$. One way or the other, we have defined $f_n$ as a bounded function, and we have \[ \int_{d_j} f_n \, d\m = \nu(d_j).\] By summing over all the $2^n$ intervals $d_j$, we see that property $(ii)$ in the statement of the lemma is satisfied (since $\nu$ assigns no mass to the finitely many points outside the union of the open intervals $d_j$). Property $(iii)$ is satisfied by the construction. Property $(i)$ also holds. Indeed, if $g$ is the characteristic function of one of the dyadic intervals $d_j$ from some stage of our construction, then the nesting property of the dyadic system and the additivity of $\nu$ ensure that \[ \lim_{n \to \infty} \int_\T g f_n \, d\m = \nu(d_j) = \int_\T g\, d\nu.\] The above equalities hold also for functions $g$ which are finite linear combinations of characteristic functions of dyadic intervals. Since such linear combinations can be used to uniformly approximate any continuous function on $\T$, and since we have the uniform variation bound in $(ii)$, we conclude that the sequence \{$f_n \, d\m\}_n$ converges weak-star to $\nu$. The proof is complete in the case that $\nu$ assigns no mass to singletons.

    In the contrary case we have that \[\nu = \sum_{j} c_j \delta_{x_j}, \quad c_j > 0\] is a countable linear combination of unit masses $\delta_{x_j}$ at the sequence of points $\{x_j\}_j$ in $\T$. Our assumption implies that $x_j \not\in \co (w) $ for all $j$. Thus each $x_j$ is the midpoint of an interval $I$ which can be chosen to have arbitrarily small length and for which we have $\int_I \log^+(1/w) \, d\m = + \infty$. We can then proceed in an analogous way to the above, and produce at each stage $n$ of the construction a disjoint finite sequence of intervals $\{I_{n,j}\}_{j=1}^n$ each covering a different point $x_j$ for $j = 1, \ldots,n$. We then define a positive function $f_n$ which carries appropriate amount of mass on each of the intervals $I_{n,j}$ and satisfies the other needed properties. We skip laying out the straight-forward details of this adaptation of the previous argument.

    The general case follows by decomposing a measure $\nu$ into one measure which is a sum of point masses and one measure which assigns no mass to singletons.
\end{proof}

\subsection{Proof of \thref{CyclicityMainTheorem}} 

We will need one more elementary lemma. It appears in \cite{el2012cyclicity} and many other works.

\begin{lem} \thlabel{ProdOfCyclicLemma} Assume that $H$ is a Banach space of analytic functions in $\D$ which contains $H^\infty$ and with the property that for all functions $h$ in $H^\infty$ the operator $M_h f := h f$ is bounded on $H$. Then the product $uv$ of two cyclic bounded functions $u,v \in H^\infty$ is cyclic.
\end{lem}

By cyclicity of $u$ we mean, of course, that there exists a sequence of analytic polynomials $\{p_n\}_{n \geq 1}$ such that $p_n u$ converges to $1 \in H$ in the norm of the space.

\begin{proof}
If $u$ and $v$ are two cyclic bounded functions, then for any polynomials $p,q$ we have the inequality \[ \|1 - puv\|_{H} \leq \|1 - qv\|_{H} + \|M_v\| \|q -  pu\|_{H},\] where $\|M_v\|$ denotes the operator norm of the multiplication operator $M_v$, and $\| \cdot \|_H$ denotes the norm in $H$. We use cyclicity of $v$ to choose the polynomial $q$ to make the first term on the right arbitrarily small, and next we use cyclicity of $u$ to choose $p$ to make the second term on the right arbitrarily small. It follows that the product $uv$ of two bounded cyclic functions is a cyclic function.    
\end{proof}

\thref{ProdOfCyclicLemma} applies to any irreducible space $\Po^2(\mu)$ of the form considered here, since indeed the multiplication by any function in $H^\infty$ induces a bounded operator on these spaces. We skip the straight-forward proof, which can for instance be based on simple analysis of the dilations $h_r(z) := h(rz)$, $ r \in (0,1)$, of the bounded function $h$. In particular, $H^\infty \subset \Po^2(\mu)$ whenever the latter is irreducible. For future reference, note that as a subspace of $\L^2(\mu)$ (with $\mu$ as in \eqref{MuGeneralStructureEq}), each function $h \in H^\infty \subset \Po^2(\mu)$ is defined also on $\mu_\T := w \, d\m$, the part of $\mu$ living on the circle $\T$. It is not hard to see that the values of $h$ with respect to $\mu_\T$ coincide with the non-tangential boundary function of $h$ on $\T$. If $w$ is bounded, then the same conclusions hold also for any $h \in H^2 \subset \Po^2(\mu)$.

\begin{proof}[Proof of \thref{CyclicityMainTheorem}]

Note first that $(i) \Rightarrow (ii)$ in \thref{CyclicityMainTheorem}, since the condition $\nu\big( \co (w) \big) > 0$ implies that a factor in $S_\nu$ satisfies the permanence property exhibited in \thref{PermanenceMainTheorem}, and so $S_\nu$ cannot by cyclic. Thus it suffices for us to show the implication $(ii) \Rightarrow (i)$. The norms induced by measures $\mu$ satisfying \eqref{ExpDecTag} are largest if the measure $\mu$ has the form in \eqref{T1} defined in Section \ref{introsec}, with $\beta = 1$. If $S_\nu$ is cyclic in $\Po^2(\mu)$ defined by any measure this form, then it is cyclic in any $\Po^2(\mu)$-space considered in this article. Thus it suffices to prove the theorem in the case of $\mu$ being of the form \eqref{T1} with $\beta = 1$, and any $c > 0$. 

Let us then assume that $\nu\big( \co (w) \big) = 0$. The formula \eqref{SingInnerDef} shows that \[S_{\nu}(z) =  \prod_{i=1}^N S_{\nu/N}(z), \quad z \in \D \] for any positive integer $N \geq 1$. Then by replacing $\nu$ by $\nu/N$ for $N$ sufficiently large, and by \thref{ProdOfCyclicLemma}, we may assume that $\nu(\T) < c/10$. Let $\{f_n\}_{n \geq 1}$ be a sequence of positive bounded functions given by \thref{WeakStarSeqObstacleLemma} for which the measures $\{f_n \, d\m\}_{n \geq 1}$ converge weak-star to $2 \nu$, which satisfy $\int_\T f_n \, d\m = 2\nu (\T)$, and for which the bound \eqref{ObstacleForfn} holds. Construct the outer functions \[ h_n(z) := \exp\Big( H_{f_n}(z)/2 \Big), \quad z \in \D,\] where \[ H_{f_n}(z) := \int_\T \frac{x + z}{x - z} f_n(x) \, d\m(x), \quad z \in \D\] is the usual Herglotz integral of $f_n$. Then, since $|H_{f_n}(z)| \leq \frac{4\nu(\T)}{1-|z|}$, we obtain \[ |h_n(z)| \leq \exp\Big( \frac{2\nu(\T)}{1-|z|}\Big) \leq \exp \Big( \frac{c}{5(1-|z|)}\Big), \quad z \in \D,\] and from property $(iii)$ in \thref{WeakStarSeqObstacleLemma} and basic properties of Herglotz integrals, we have the non-tangential boundary value estimate \[|h_n(x)| = \exp \Big( f_n(x)/2\Big) \leq \sqrt{\max[1, 1/w(x) ]}\] for almost every $x \in \T$ with respect to $m$. It follows from these inequalities and the definition of the norm in $\Po^2(\mu)$ that the family $\{h_n\}_{n \geq 1} \subset H^\infty$ forms a bounded subset of the Hilbert space $\Po^2(\mu)$. Moreover, by the weak-star convergence of $\{f_n\, d\m\}_{n \geq 1}$ to $2 \nu$ we have that \[ \lim_{n \to \infty} h_n(z) = \frac{1}{S_{\nu}(z)}, \quad z \in \D.\] But this means that $1/S_\nu$ is a member of $\Po^2(\mu)$, since we can identify it as a weak cluster point of some subsequence of $\{h_n\}_{n \geq 1}$. Thus there must also exist a sequence of polynomials $\{p_n\}_n$ tending to $1/S_\nu$ in the norm of $\Po^2(\mu)$. Consequently, since the multiplication operator $M_{S_\nu}$ is a bounded on our space, we have that $S_\nu p_n \to 1$ in the norm of $\Po^2(\mu)$. That is, $S_\nu$ is cyclic. 
\end{proof}

\section{Moment functions, admissible sequences and spaces of Taylor series}
\label{MomentEstSection} 

This section initiates the second part of the article. In this part, we will apply our previous results in $\Po^2(\mu)$-theory to Cauchy integrals, model spaces and the de Branges-Rovnyak spaces $\hb$. In order to do so, we will need to analyze the moments of the functions $G$ appearing in \eqref{MuGeneralStructureEq}. This entire section is concerned with this analysis.

\subsection{Admissible sequences and their properties}

If $G$ is a function satisfying \eqref{ExpDecTag} and \eqref{LogLogIntTag}, then the sequence of moments of $G$, defined below in \eqref{GMomentDef}, will be shown to satisfy the following basic properties.

\begin{defn}{\textbf{(Admissible sequences)}} \thlabel{AdmissibleSequenceDef} A decreasing sequence of positive numbers $\{M_n\}_{n \geq 0}$ with \[\lim_{n \to \infty} M_n = 0\]  will be called \textit{admissible} if it satisfies the following three conditions:
\begin{enumerate}[(i)]
    \item the sequence $\{\log M_n\}_{n \geq 0}$ is eventually convex, in the sense that \[ 2\log M_n \leq \log M_{n+1} + \log M_{n-1}\] for all sufficiently large $n \geq 0$,
    \item there exists $d > 0$ such that \[ M_n \leq \exp ( - d \sqrt{n})\] for all sufficiently large $n \geq 0$,
    \item the summability condition\[ \sum_{n \geq 0} \frac{\log (1/M_n)}{1+n^2} < \infty\] is satisfied. 
\end{enumerate}
\end{defn}

With later applications in mind, it will be useful to single out the following simple preservation property of admissible sequences under taking powers.

\begin{prop}
    \thlabel{PowerAdmSeq}
    If $M = \{M_n\}_{n \geq 0}$ is an admissible sequence, then so is \[M^p :=\{M_n^p\}_{n \geq 0},\] for any $p > 0$.
\end{prop}

The proposition follows immediately from \thref{AdmissibleSequenceDef}

\subsection{Legendre envelopes}

Roughly speaking, admissible sequences $\{M_n\}_{n\geq 0}$ are in a correspondence with moments of functions $G$ satisfying \eqref{ExpDecTag} and \eqref{LogLogIntTag}, and we shall now proceed to make this statement more precise. In order to do so, we will need to recall some basic concepts from convex analysis. In particular we will use the notion of Legendre envelopes and their properties. In parts of our exposition we will follow Beurling in \cite{beurling1972analytic} and Havin-Jöricke in \cite{havinbook}, and we will refer to those works for most of the proofs of the following claims.

Let $m(x)$ be a positive and continuous function defined for $x > 0$, which is decreasing and satisfies \[\lim_{x \to 0^+} m(x) = +\infty.\] In our application, $m$ will be of the form $m(x) = \log 1/G(x)$ (for small $x$). The \textit{lower Legendre envelope} $m_*$ is defined as 
\begin{equation}
    \label{lowerLegendreEnvelopDef} m_*(x) := \inf_{y > 0} \, m(y) + xy, \quad x > 0.
\end{equation}
Being an infimum of concave (actually affine) and increasing functions, $m_*$ is itself concave and increasing, and it is easy to see that \[ \lim_{x \to +\infty} m_*(x) = +\infty.\] 

\begin{remark} \thlabel{remarkLowerLegEnvelope} Assume that we modify the function $m$ above for $x$ larger than $1$, so that we end up with a different function $\widetilde{m}$ which satisfies $\widetilde{m}(x) = m(x)$ for $x < 1$, but the values of the two functions might differ for $x \geq 1$. Then it is not hard to see from the definition in \eqref{lowerLegendreEnvelopDef} that $\widetilde{m}_*(x) = m_*(x)$ for all sufficiently large $x$. Indeed, if $y \geq 1$, then we have by positivity of $m$ that \[  m(y) + xy \geq x \geq m(1/2) + x/2,\] the second inequality holding if $x$ is sufficiently large. For such $x$, the candidate $y = 1/2$ is always better than any candidate $y \geq 1$ in the infimum in \eqref{lowerLegendreEnvelopDef}, and our claim follows.
\end{remark}

In \cite[Lemma 1]{beurling1972analytic}, Beurling proves the following statement which will be used below.
\begin{prop}\thlabel{LogIntLogSumEquivalenceProp}
Let $m(x)$ be a positive, continuous and decreasing function of $x > 0$ which satisfies $\lim_{x \to 0^+} m(x) = +\infty.$ The following two statements are equivalent:
\begin{enumerate}[(i)]
    \item there exists a $\delta > 0$ such that \[\int_0^\delta \log m(x) \, dx < \infty, \]
    \item we have \[ \int_{1}^\infty \frac{m_*(x)}{1+x^2} dx < \infty.\]
\end{enumerate}    
\end{prop}

We refer the reader to \cite{beurling1972analytic} for a proof of \thref{LogIntLogSumEquivalenceProp}.

Let $k(x)$ be a positive concave function of $x > 0$ which is increasing and satisfies \[ \lim_{x \to +\infty} k(x) = +\infty .\] We will consider its \textit{upper Legendre envelope} defined as 
\begin{equation}
    \label{upperLegendreEnvelopDef} k^*(x) := \sup_{y > 0} \, k(y) - xy.
\end{equation}
Then it is easy to see that $k^*$ is a convex and decreasing function, and \[\lim_{x \to 0^+} k^*(x) = +\infty.\] We have the following inversion formula, which is well-known (see \cite[p. 224-225]{havinbook}).

\begin{lem} \thlabel{LegendreInversionFormula}
    Let $k$ be a positive concave function of $x > 0$ which is increasing and satisfies $\lim_{x \to \infty} k(x) = +\infty$. Then \[(k^*)_*(x) = k(x).\] 
\end{lem}

%\begin{proof}
   % Note first that for any $x > 0$ and $y > 0$, we have from \eqref{upperLegendreEnvelopDef} that \[ k^*(x) \geq k(y) - xy, \] so that \[ k^*(x)+xy \geq k(y).\] Taking the infimum in $x > 0$ of the right-hande side of the above expression tells us, in accordance with \eqref{lowerLegendreEnvelopDef}, that $(k^*)_*(y) \geq k(y)$ for all $y > 0$. For the reverse inequality, recall that concavity of $k$ means that $k(y)$ is the infimum of the values $ay+b$ over pairs $(a,b) \in \R^2$ such that the line $\{ at+b : t > 0\}$ lies above the graph of $k$. Since $k$ is increasing to $+\infty$, for any such pair we must have $a > 0$, and clearly \[ at + b \geq k(t)  \quad \Leftrightarrow \quad b \geq k(t) - at, \quad t > 0,\] which by taking supremum in $t$ and using definition \eqref{upperLegendreEnvelopDef} means that $b \geq k^*(a)$. Given $\epsilon > 0$, let $(a, b)$ be such that $ay+b \leq k(y) + \epsilon$ and the line $\{ at + b : t > 0\}$ lies above the graph of $k$. We have just shown that $b \geq k^*(a)$. Therefore    
    %\begin{align*}
     %   k(y) + \epsilon &\geq ay + b \\
   %     & \geq ay + k^*(a) \\ 
    %    & \geq \inf_{a > 0} \, k^*(a) + ay \\
    %    &= (k^*)_*(y),
    %\end{align*} and so the proof is completed by letting $\epsilon$ tend to zero.
%\end{proof}

\subsection{A characterization of admissible sequences} 
We will be interested in the sequence $\{M_n\}_{n \geq 0}$ of moments of the parts of our measures $\mu$ living on $\D$: 
\begin{align} \label{GMomentDef}
    M_n = M_n(G) &:= \int_\D G(1-|z|) |z|^{2n} dA(z) \nonumber \\
    &= 2 \int_0^1 G(1-r) r^{2n+1} dr.
\end{align}
We define the \textit{moment function} of $G$ by
\begin{equation}
    \label{PMomentFuncDef}
    P_G(x) := \int_0^1 G(1-r) r^x \, dx, \quad x > 0.
\end{equation}

The next lemma gives an estimate on $P_G$. We skip the proof, which is essentially the same as the one given in \cite[p. 229]{havinbook} (see also the proof of \cite[Lemma 4.3]{malman2023revisiting}).

\begin{lem}
    \thlabel{MomentGGrowthLogLogInt} Let $G(x)$, $x \in [0,1]$, be an increasing continuous function satisfying $G(0) = 0$, and put
    \begin{equation}
        m(x) := 
        \begin{cases}
            \log 1/G(x), &  x \in (0,1] \\
            \log 1/G(1), & x > 1. 
        \end{cases}
    \end{equation}
    Then, for sufficiently large $x > 0$, we have the estimates 
    \[ \frac{\exp \big( -m_*(2x) \big) }{4x} \leq P_G(x) \leq \exp \big(-m_*(x)\big),\] where $P_G$ is the moment  function of $G$ defined in \eqref{PMomentFuncDef}.   
\end{lem}

\begin{lem} \thlabel{MomentEstProp} For $c > 0$ and $\beta > 0$, let $\{ M_n(\beta, c) \}_{n \geq 0}$ be the sequence of moments given by \begin{align}
    \label{MomentsT1Eq} M_n(\beta, c) &:= \int_\D \exp\Big(- c(1-|z|)^{-\beta}\Big)|z|^{2n} \, dA(z) \nonumber \\ & = 2 \int_0^1 \exp\Big(- c(1-r)^{-\beta}\Big) r^{2n+1} \, dr, \quad n \geq 0.
\end{align} Put \[ \tilde{\beta} := \frac{\beta}{\beta+1}.\] For sufficiently large positive $n$, we have the estimates 
\begin{equation}
    \label{d0d1MomentEq} \exp\Big(-2d n^{\tilde{\beta}}\Big) \leq M_n(\beta, c) \leq \exp \Big(-d n^{\tilde{\beta} }\Big)
\end{equation} where $d = d(\beta,c)$ is comparable to $c^{1/(\beta+1)}$ if $\beta$ remains fixed.
\end{lem}

\begin{proof} In the notation of \thref{MomentGGrowthLogLogInt}, and with \[G(x) = \exp \Big(-\frac{c}{x^\beta}\Big), \quad x \in (0,1)\] we have \[ m(x) = \frac{c}{x^\beta}, \quad x \in (0,1)\] and we need to compute the corresponding Legendre envelope $m_*$ defined in \eqref{lowerLegendreEnvelopDef}. Having fixed some number $x > 0$, we use elementary calculus to show that \[\inf_{y > 0} \, \frac{c}{y^\beta} + xy\] is attained at the point \[y_* := \Bigg( \frac{c\beta}{x}\Bigg)^{\frac{1}{\beta + 1}}\] from which it follows that \begin{align*}
    m_*(x) &= cy_*^{-\beta} + xy_* \\
            &= c^{1/(\beta+1)}(\beta^{-\beta/(\beta+1)} + \beta^{1/(\beta+1)}) x^{\widehat{\beta}} \\
            & := d(\beta, c) x^{\widehat{\beta}}.
\end{align*} Since \[M_n(\beta,c) = 2 P_G(2n+1)\] we obtain from \thref{MomentGGrowthLogLogInt} the inequalities 
\begin{equation}
    2\exp\big (-d(\beta,c)(4n+2)^{\tilde{\beta}} - \log(8n+4)\big) \leq M_n(\beta,c) \leq 2\exp\big (-d(\beta,c)(2n+1)^{\tilde{\beta}} \big) 
\end{equation} which hold for all sufficiently large $n$. Our result follows easily from this.
\end{proof}

We can now prove the main result of the section, which connects our considered class of functions $G$ with the admissible sequences appearing in \thref{AdmissibleSequenceDef}.

\begin{prop} \thlabel{AdmissibleSequenceLemma} 
If $G$ satisfies \eqref{ExpDecTag} and \eqref{LogLogIntTag}, then $\{M_n\}_{n \geq 0}$ defined by \begin{equation}
    \label{AdmSeqGMomPropEq}
M_n := 2 \int_0^1 G(1-r) r^{2n+1} \, dr
\end{equation} is an admissible sequence. 

Conversely, if $\{M_n\}_{n \geq 0}$ is an admissible sequence, then there exists a continuous and increasing function $G$ satisfying \eqref{ExpDecTag}, \eqref{LogLogIntTag}, $G(0) = 0$, and for which the inequality\[P_G(2n+1) \leq M_n\] holds for all sufficiently large $n \geq 0$.
\end{prop}

\begin{proof} We start by proving that the sequence in \eqref{AdmSeqGMomPropEq} is admissible by verifying the three conditions in \thref{AdmissibleSequenceDef}. By the Cauchy-Schwarz inequality, we have \begin{align*}
    M_n & = 2 \int_0^1 G(1-r) r^{(n-1) + 1/2}r^{(n+1) + 1/2} \, dr 
    \\ & \leq \sqrt{M_{n-1}}\sqrt{M_{n+1}}
\end{align*} Thus $\{ \log M_n \}_{n \geq 0}$ is a convex sequence. The inequality $M_n \leq \exp(-c \sqrt{n})$ for some $c > 0$ and all sufficiently large $n \geq 0$ follows readily from \eqref{ExpDecTag} and an application of the upper estimate \thref{MomentEstProp} with $\beta = 1$ (and consequently $\tilde{\beta} = 1/2$). Let $m$ be as in \thref{MomentGGrowthLogLogInt}. By the lower estimate in that lemma, we have \begin{align*}
    \sum_{n \geq 0} \frac{\log 1/M_n}{1+n^2} &= \sum_{n \geq 0} \frac{-\log 2 - \log P_G(2n+1)}{1+n^2} \\
    & \leq \sum_{n \geq 0} \frac{-\log 2 + m_*(4n+2) + \log (8n+4)}{1+n^2}.
\end{align*} The assumption that $G$ satisfies \eqref{LogLogIntTag} implies that $\int_0^1 \log m(x) \, dx < \infty$, and so from \thref{LogIntLogSumEquivalenceProp} we deduce that the last sum above is convergent. Consequently, $\{M_n\}_{n \geq 0}$ is an admissible sequence.

Conversely, assume that $\{M_n\}_{n \geq 0}$ is an admissible sequence. Since the sequence tends to zero, we may without loss of generality assume that $M_0 < 1$. From property $(i)$ in \thref{AdmissibleSequenceDef} we obtain the inequality \[ \log 1/M_{n+1} - \log 1/M_n \leq \log 1/M_n - \log 1/M_{n-1}, \quad n \geq 0.\] This means that the slopes of the line segments between each consecutive pair of points in the sequence
\begin{equation}
    \label{kData} (2n+1, \log 1/M_n), \quad n \geq 0
\end{equation} are decreasing, which means that if we define the function $k(x)$, $x > 0$, as the piecewise linear interpolant of the data \eqref{kData}, then $k$ is concave, continuous, positive and increasing, and satisfies \[k(2n+1) = \log 1/M_n, \quad n \geq 0.\] It also satisfies $\lim_{x \to \infty} k(x) = +\infty$, and property $(iii)$ in \thref{AdmissibleSequenceDef} easily implies that \begin{equation}
    \label{kIntEq}
    \int_1^\infty \frac{k(x)}{1+x^2} \, dx < \infty.
\end{equation}
Let $k^*$ be the upper Legendre envelope of $k$ defined in \eqref{upperLegendreEnvelopDef}, set 

\begin{equation}
    \label{GdefSeqLemmaProof} G(x) := \exp \big(- k^*(x) \big), \quad  x \in (0,1],
\end{equation} and $G(0) = 0$. Then $G$ is a continuous and increasing function. Define $P_G$ as in \eqref{PMomentFuncDef}. By \thref{remarkLowerLegEnvelope}, \eqref{GdefSeqLemmaProof}, inversion formula in \thref{LegendreInversionFormula} and \thref{MomentGGrowthLogLogInt}, we have the estimate \[P_G(x) \leq \exp\big(-k(x)\big)\] for all sufficiently large $x$. Consequently, \[ P_G(2n+1) \leq M_n\] if $n$ is large, since $k$ interpolates the data \eqref{kData}. \thref{LogIntLogSumEquivalenceProp}, \thref{LegendreInversionFormula} and \eqref{kIntEq} imply that \[ \int_0^1 \log \log 1/G(x) \, dx = \int_0^1 \log k^*(x) \, dx < \infty.\] Thus $G$ satisfies \eqref{LogLogIntTag}. It remains to check that $G$ also satisfies \eqref{ExpDecTag}. Note that property $(ii)$ in \thref{AdmissibleSequenceDef} of the admissible sequence $\{M_n\}_{n \geq 0}$ implies easily that $k$ satisfies a lower bound of the form \[ k(x) \geq d \sqrt{x}, \quad x \geq 0,\] for some constant $d > 0$. But then, by \eqref{upperLegendreEnvelopDef}, we have 
\begin{align*}
    k^*(x) &= \sup_{ y \geq 0} \, k(y) -xy \\ & \geq \sup_{y > 0} \, d \sqrt{y} - xy \\ &= \frac{d^2}{4x}.
\end{align*} The last equality can be derived by elementary calculus techniques. Consequently \[ \liminf_{x \to 0^+} x \log 1/G(x) \geq \frac{d^2}{4} > 0,\] and so $G$ satisfies \eqref{ExpDecTag}. The proof is complete.
\end{proof}

\subsection{Some auxiliary spaces of Taylor series} \label{AuxHilbSec}

If $f: \D \to \mathbb{C}$ is an analytic function and 
\begin{equation}
    \label{muRestricDiskT1}
    d\mu_\D(z) =  G(1-|z|) dA(z)
\end{equation} then we have the norm equality

\begin{equation} \label{MuDIsomorphicNorm}
\|f\|^2_{\mu_\D} = \int_\D |f(z)|^2 d\mu_\D(z) = \sum_{n \geq 0} M_{n}(G)|f_n|^2 
\end{equation} where $\{ f_n\}_{n \geq 0}$ is the sequence of Taylor coefficients of $f$, and $M = \{M_n(G)\}_{n \geq 0}$ is given by \eqref{GMomentDef}. The above equality gives us an isometric isomorphism between $\Po^2(\mu_\D)$ and a space of Taylor series.

For a decreasing sequence $M = \{ M_n \}_{n \geq 0}$ of positive numbers we define $H_2(M)$ to be the Hilbert space of analytic functions in $\D$ consisting of $f(z) = \sum_{n \geq 0} f_n z^n$ which satisfy \begin{equation}
     \label{HbetacDef} \|f\|^2_{H_2(M)} := \sum_{n \geq 0} M_n |f_n|^2 < \infty.
\end{equation} 
In our development, the sequences $M$ will be the admissible sequences studied in Section \ref{MomentEstSection}. Such sequences have the property that \[ \lim_{n \to \infty} M_n^{1/n} = 1,\] a condition which ensures that the spaces $H_2(M)$, and their duals, are genuine spaces of analytic functions on $\D$. The dual space $H_2^*(M)$ is to consist of power series which satisfy \begin{equation}
     \label{HbetacDefDual} \|f\|^2_{H_2^*(M)} := \sum_{n \geq 0} \frac{|f_n|^2}{M_n} < \infty.
\end{equation} 
Since $M = \{M_n\}_{n \geq 0}$ is assumed to be decreasing, the space $H_2^*(M)$ is contained in the Hardy space $H^2$. In fact, if $M$ is an admissible sequence, then $H^*_2(M)$ consists of functions satisfying the condition \eqref{ESDpropEq} of Section \ref{introsec}. The duality between $H_2(M)$ and $H_2^*(M)$ is realized by the usual Cauchy pairing 
\begin{equation}
    \label{CauchyDualityHM}
\ip{f}{g} := \lim_{r \to 1^-} \sum_{n \geq 0} r^{2n}f_n \conj{g_n} = \int_\T f\conj{g}\,  d\m = \ip{f}{g}_{\L^2} \end{equation} where the sequential definition above makes sense whenever $f \in H_2(M)$, $g \in H_2^*(M)$, and the integral definition holds only in special cases, for instance when $f,g \in H^2$. An application of the Cauchy-Schwarz inequality to the limit in \eqref{CauchyDualityHM} shows that \[ |\ip{f}{g}| \leq \|f\|_{H_2(M)}\|g\|_{H^*_2(M)}.\] 
We introduce also the space $H_1^*(M)$ which consists of power series $f(z) = \sum_{n \geq 0} f_nz^n$ satisfying \begin{equation}
     \label{Hstar1Def} \|f\|_{H^*_1(M)} := \sup_{n \geq 0} \, \frac{|f_n|}{M_n} < \infty. 
\end{equation} 
Recall from \thref{PowerAdmSeq} that the family of admissible sequences introduced in \thref{AdmissibleSequenceDef} is invariant under taking powers. For this reason, the spaces $H^*_2(M)$ and $H^*_1(M)$ which appear in our study are very similar. 

\begin{lem} \thlabel{H1starH2starEmbeddingLemma}
    Let $M = \{M_n\}_{n \geq 0}$ be an admissible sequence, and consider the sequences \[ M^p := \{M^p_n\}_{n \geq 0}.\] For $p > 1/2$, we have the continuous embeddings \[ H^*_1(M^p) \subset H^*_2(M) \subset H^*_1(M^{1/2}).\]
\end{lem}

\begin{proof}
    If $f \in H^*_2(M)$, then for any $n \geq 0$ we have that \[ \frac{|f_n|^2}{M_n} \leq \|f\|^2_{H^*_2(M)},\] so clearly $f \in H^*_1(M^{1/2})$. If $f \in H^*_1(M^p)$ for some $p > 1/2$, then we may use that $M$ satisfies property $(ii)$ of admissible sequences in \thref{AdmissibleSequenceDef} to obtain \begin{align*}
        \sum_{n \geq 0} \frac{|f_n|^2}{M_n} &\leq \|f\|^2_{H^*_1(M^p)} \sum_{n \geq 0} M_n^{2p-1} \\
        &\leq \sum_{n \geq 0} \exp \big( -d(2p-1) \sqrt{n}\big) \\ &< \infty.
    \end{align*}
    Thus $f \in H^*_2(M)$.
\end{proof}

The following corollary will be used several times below. 

\begin{cor} \thlabel{MhstarContainmentFromRSD} If an analytic function $f$ in $\D$ satisfies the condition \eqref{ESDpropEq}, then there exists $c' > 0$ and a measure \[ d\mu_\D(z) =  \exp\Big(- \frac{c'}{(1-|z|)}\Big) dA(z) = G(1-|z|) dA(z) \] with sequence of moments $M = \{M_n(G)\}_{n\geq 0}$ such that $f \in H^*_2(M)$.
\end{cor}

\begin{proof}
The condition \eqref{ESDpropEq} and \thref{H1starH2starEmbeddingLemma} imply that $f \in H^*_2(\widetilde{M})$, where $\widetilde{M}_n = \exp(-c_0 \sqrt{n})$ for some positive constant $c_0$. Now \thref{MomentEstProp}, with $\beta = 1$, shows that $c' > 0$ can be chosen so that $M_n(G) = M(\beta, c') \geq \widetilde{M}_n$ for sufficiently large $n$. Then $f \in H^*_2(M)$.
\end{proof}
 
We end the section with a few words about operators acting on the introduced class of spaces. From their definition, and in particular from the assumption on $M$ being decreasing, it is not hard to see that the spaces $H_2(M)$ are invariant under the multiplication operator $M_z$, and that this operator is a contraction. Then Von Neumann's inequality (\cite[p. 159]{mccarthypick}) or the Sz.-Nagy Foias $H^\infty$-functional calculus (\cite[Chapter 3]{nagyfoiasharmop}) shows that in fact every function $h \in H^\infty$ defines a bounded multiplication operator $M_h: H_2(M) \to H_2(M)$. The adjoint operator $M_h^*: H_2^*(M) \to H_2^*(M)$ is easily indentified with the usual Toeplitz operator $T_{\conj{h}}$ with the co-analytic symbol $\conj{h}$, i.e., $T_{\conj{h}}f$ is the orthogonal projection to the Hardy space $H^2$ of the function $\conj{h}f \in \L^2(\T)$.

For later reference, we record these observations in a proposition.

\begin{prop} \thlabel{HToeplitzInvariance} Let $M = \{M_n\}_{n \geq 0}$ be an admissible sequence.
\begin{enumerate}[(i)]
    \item The space $H_2(M)$ is invariant for the analytic multiplication operators \[M_hf = h(z)f(z)hf, \quad f \in H_2(M),\] with symbols $h \in H^\infty$.
    \item The space $H^*_2(M)$ is invariant for the co-analytic Toeplitz operators \[T_{\conj{h}}f = P_+\conj{h}f, \quad f \in H^*_2(M)\] with symbols $h \in H^\infty$.
\end{enumerate}
\end{prop}

\begin{cor}\thlabel{ToeplitzInvRSD} If an analytic function $f: \D \to \mathbb{C}$ satisfies the condition \eqref{ESDpropEq}, then so does $T_{\conj{h}}f$ for any $h \in H^\infty$.
\end{cor}
\begin{proof}
    We use \thref{MhstarContainmentFromRSD} and \thref{HToeplitzInvariance} to see that the function $T_{\conj{h}}f$ is contained in a space $H^*_2(M)$, where $M$ is admissible. \thref{H1starH2starEmbeddingLemma} shows that $T_{\conj{h}}f \in H^*_1(\sqrt{M})$, so $T_{\conj{h}}f$ satisfies \eqref{ESDpropEq}.
\end{proof}

\section{Existence in $\hb$ of functions with rapid spectral decay}

\label{ExistenceESDHbSec}

This section deals with proving \thref{MainTheoremHbExistenceESD}. In the proof, we will need a similar result in the context of model spaces, which we establish first. Next, we present some background theory of $\hb$-spaces which will be needed in the proof of \thref{MainTheoremHbExistenceESD}, and also in the proof of \thref{MainTheoremHbDensityESD} given in the next section.

\subsection{Corresponding result in model spaces}
\label{ExtremalDecModelSpaceSec}

The following \thref{AlphaKthetaBreakpointProp} needed in the proof of \thref{MainTheoremHbExistenceESD} is known, and follows for instance from the work of Beurling in \cite{beurling1964critical}, or from a result of El-Fallah, Kellay and Seip in \cite{el2012cyclicity}. The mentioned results are much stronger than \thref{AlphaKthetaBreakpointProp}. Because the result is important for our further purposes, we shall use the estimates from Section \ref{MomentEstSection} to give a simple proof of our version of the result.

\begin{prop} \thlabel{AlphaKthetaBreakpointProp} If $\theta$ is a singular inner function, then the model space $\K_\theta$ contains no non-zero function $f(z) = \sum_{n \geq 0} f_nz^n$ which satisfies \eqref{ESDpropEq}.
\end{prop}

\begin{proof}
We will show that any $f \in K_\theta$ which satisfies \eqref{ESDpropEq} satisfies also $f(0) = f_0 = 0$. Since $\K_\theta$ is invariant for the backward shift \[Lf(z) := \frac{f(z) - f(0)}{z}, \quad z \in \D,\] and by \thref{ToeplitzInvRSD} the function $Lf = T_{\conj{z}}f$ satisfies \eqref{ESDpropEq}, the same argument will show that $f_n = L^nf(0) = 0$ for $n \geq 0$. Thus $f \equiv 0$ will follow.

We apply \thref{MhstarContainmentFromRSD} to $f$ and obtain a measure $\mu_\D$ with moment sequence $M$ such that $f \in H^*_2(M)$. The measure $\mu_\D$ is of the form \eqref{T1} (see Section \ref{introsec}) for $\beta = 1$ and $w \equiv 0$. By \thref{CyclicityMainTheorem}, the singular inner function $\theta$ is trivially cyclic in $\Po^2(\mu_\D)$, since $\co ( w ) = \co ( 0 ) = \varnothing$. Thus there exists a sequence of analytic polynomials $\{p_n\}_{n \geq 0}$ such that $\theta p_n \to 1$ in the norm of $\Po^2(\mu) = H_2(M)$. Using the duality pairing \eqref{CauchyDualityHM} and the membership of $f$ in $H^*_2(M) \cap K_\theta$, the following computation is justified: 
\begin{align*}
    \conj{f(0)} &= \ip{1}{f} \\
        & = \lim_{n \to \infty} \ip{\theta p_n}{f} \\
        & = \lim_{n \to \infty} \int_\T \theta p_n\conj{f} \, d\m \\
        & = 0.
\end{align*} The last equality holds due to $f$ being a member of $K_\theta = (\theta H^2)^{\perp}$. Thus $f(0) = 0$, and the proof is complete by the initial remarks.
\end{proof}

\subsection{Some $\hb$-theory}

\label{HbSecAppendix}

The following description of $\hb$-spaces is very convenient in connection with various functional-analytic arguments. It has been introduced in \cite{dbrcont}, and was later used in \cite{limani_malman_2023} and \cite{limani2023problem}, to prove approximation results in classes of $\hb$-spaces. We will employ it in a similar way below. Recall that the symbol $P_+$ denotes the orthogonal projection operator $P_+ : \L^2(\T) \to H^2$, and $\L^2(E)$ denotes the subspace of those $g \in \L^2(\T)$ which live only on the measurable subset $E \subset \T$.

\begin{prop} \thlabel{normformula}
Let $b$ be an extreme point of the unit ball of $H^\infty$, 
\begin{equation}
    \label{DeltabDef}
    \Delta_b(x) = \sqrt{1-|b(x)|^2}, \quad x \in \T,
\end{equation} and $E$ be the carrier set of $\Delta_b$:

\[ E = \{ x \in \T : \Delta_b(x) > 0 \}.\] Then $f \in H^2$ is a member of $\hb$ if and only if the equation 
\begin{equation} \label{hbconteq} P_+ \conj{b}f = -P_+ \Delta_b g \end{equation} has a solution $g\in \L^2(E)$. The solution is unique, and the map $J:\Hb\to H^2\oplus \L^2(E)$  defined by $$Jf=(f,g),$$ 
is an isometry. Moreover,  \begin{equation} \label{Jort} J(\Hb)^\perp = \Big\{ (bh, \Delta_b h) : h \in H^2 \Big\}. \end{equation}
\end{prop}

Next comes a very useful corollary which is well-known and can be proved by other means (see \cite{hbspaces1fricainmashreghi}, \cite{hbspaces2fricainmashreghi} for other derivations). 

\begin{cor} \thlabel{CauchyTransformsinHb} Let $E$ and $\Delta_b$ be as in \thref{normformula}. For any $s \in \L^2(E)$, the function \[f = P_+ \Delta_b s\] is a member of $\hb$ and, in the notation of \thref{normformula}, we have \[ Jf = (f, -\conj{b}s).\]
Moreover, if $b$ is extreme and $s$ is non-zero, then $f$ is non-zero. 
\end{cor}

\begin{proof}
    We compute \[ P_+\conj{b} f = P_+ \conj{b} P_+ \Delta_b s = P_+ \conj{b} \Delta_b s = P_+ \Delta_b \conj{b} s,\] and so \eqref{hbconteq} holds for the pair $(f,g) := (P_+\Delta_b s, -\conj{b} s)$. If $b$ is extreme, then $\log \Delta_b \not\in \L^1(\T)$, and it follows readily that also $\log( \Delta_b |s| ) \not\in \L^1(\T)$. A function $h \in \ker P_+$ is conjugate-analytic, and so $\log |h| \in \L^1(\T)$ if $h \neq 0$. So $\Delta_b s \not\in \ker P_+$ if $s$ is non-zero, and it follows that $f$ is non-zero.
\end{proof}

\begin{cor} \thlabel{TconjbInvariance} The Toeplitz operator $T_{\conj{b}}$ acts boundedly on $\hb$. If $f \in \hb$ and $Jf = (f,g)$, then \[T_{\conj{b}} f = (T_{\conj{b}} f, \conj{b} g).\]
\end{cor}

\begin{proof}
    Again, we only need to verify that \eqref{hbconteq} holds for the given pairs. This follows easily by applying the operator $T_{\conj{b}}$ to both sides of the equation \eqref{hbconteq} and computing as in the proof of \thref{CauchyTransformsinHb}.
\end{proof}

\subsection{Main tool in the proof of \thref{MainTheoremHbExistenceESD}: residuals} 

We will now need to introduce the notion of residual sets.

\begin{defn}{\textbf{(Residual sets of weights)}} \thlabel{ResDef} Let $w \in \L^1(\T)$ and consider the carrier set \[ E = \{ x \in \T : w(x) > 0 \}.\] We define $\res (w)$ to be the set \[ \res (w) = E \setminus \co (w),\] where $\co (w)$ is the set appearing in \thref{CoreDef}.
\end{defn}

Since $E$ might only be defined up to a set of $m$-measure zero, the same is true for the residual $\res (w)$ of any weight $w$. This will not cause us any problems.

We have introduced the residuals because of their crucial role in the following special case of \cite[Theorem A]{malman2023revisiting}.

\begin{lem} \thlabel{ResSetMainLemma} Assume that $w \in \L^1(\T)$ is a weight for which $\res (w)$ has positive $m$-measure. Let $w_r = w|\res (w)$ be the restriction of the weight $w$ to the set $\res (w)$. Then we have the containment \[\L^2(w_r d\m) \subset \Po^2(\mu) \] whenever $\mu$ is of the form \eqref{MuGeneralStructureEq} with $G$ satisfying \eqref{ExpDecTag}.
\end{lem}

\subsection{Proof of \thref{MainTheoremHbExistenceESD}}

In \thref{MainTheoremHbExistenceESD}, it is obvious that $(ii) \Rightarrow (i)$. We can thus prove the theorem by showing validity of the implications $(i) \Rightarrow (iii)$ and $(iii) \Rightarrow (ii)$.

Let us first show that $(i) \Rightarrow (iii)$, and so we assume that $f(z) = \sum_{n \geq 0} f_nz^n \in \hb$ is non-zero and that it satisfies \eqref{ESDpropEq}. We may assume that $b$ does not vanish at any point in $\D$, else $(iii)$ certainly holds. Similarly to as it was done in the proof of \thref{AlphaKthetaBreakpointProp}, we use \thref{MhstarContainmentFromRSD} to obtain a measure
\begin{align} \label{muEqinproof}
    d\mu & = d\mu_\D + d\mu_\T \\
    & = \exp\Big( -\frac{c'}{(1-|z|)}\Big) dA(z) + \Delta_b \,d\m \nonumber,
\end{align} and a sequence $M$ such that the identity map between $\Po^2(\mu_\D)$ and $H_2(M)$ is an isometry, and $f \in H^*_2(M)$. By \thref{HToeplitzInvariance}, the space $H^*_2(M)$ is invariant under Toeplitz operators with co-analytic symbols, and consequently we also have $T_{\conj{b}}f \in H^*_2(M)$. By \thref{TconjbInvariance} and \eqref{hbconteq} we have $T_{\conj{b}}f = P_+ \conj{b} f \in \hb \cap H^*_2(M)$ and 
\begin{equation} \label{reprTbf} T_{\conj{b}}f = P_+\Delta_b g \end{equation} 
for some $g \in \L^2(E)$. The kernel of the operator $T_{\conj{b}}$ is the model space $\K_{I_b}$, where $I_b$ is the inner factor of $b$. Since $b$ does not vanish in $\D$, it follows that $I_b$ is a purely singular inner function. Every function in $H^*_2(M)$ satisfies \eqref{ESDpropEq}, so \thref{AlphaKthetaBreakpointProp} implies that $\K_{I_b} \cap H^*_2(M) = \{0\}$. Consequently $f \not\in \K_{I_b}$, so $T_{\conj{b}}f \neq 0$, and $\Delta_b g \neq 0$ by \eqref{reprTbf}. If \[ T_{\conj{b}} f(z) = \sum_{n \geq 0} c_nz^n\] is the Taylor expansion of $T_{\conj{b}} f(z)$, then a consequence of the membership $T_{\conj{b}}f \in H^*_2(M)$ is that the function \[ F(z) := \sum_{n \geq 0} \frac{c_n}{M_n} z^n, \quad z \in \D\] is a member of $H_2(M)$. The function $F$ lives on $\D$, the function $g$ lives on $\T$, and hence $F - g$ defines a function on $\cD$. The condition $F \in H_2(M)$ means simply that $F$ is square-integrable with respect to the part $\mu_\D$ of $\mu$ in \eqref{muEqinproof} which lives on $\D$. The containment $g \in \L^2(\Delta_b d\m) = \L^2(\mu_\T)$ is ensured by the boundedness of $\Delta_b$ and the containment $g \in L^2(E)$. Thus $F - g \in \L^2(\mu)$. The representation \eqref{reprTbf} tells us that the positive Fourier coefficients $\{c_n\}_{n \geq 0}$ of $T_{\conj{b}}f$ and of $\Delta_b g$ coincide. Our definitions then imply that the function $F - g$ is orthogonal to the analytic polynomials in $\L^2(\mu)$. Since $\Delta_b g \neq 0$, the function $g$ is a non-zero element of $\L^2(\mu_\T)$. The conclusion is that there exists an element (namely $F - g$) inside $\L^2(\mu)$ which is orthogonal to $\Po^2(\mu)$ and which does not vanish identically on the circle $\T$. If there existed no interval on which $\log \Delta_b$ was integrable, then $\co (\Delta_b) = \varnothing$, and so \thref{ResSetMainLemma} would imply that the entire space $\L^2(\Delta_b d\m)$ is contained in $\Po^2(\mu)$. Clearly that would be a contradiction to $F - g$ being orthogonal to $\Po^2(\mu)$. Thus such an interval exists, and we have proved that $(i) \Rightarrow (iii)$.

The implication $(iii) \Rightarrow (ii)$ is easier. Let $M = \{M_n\}_{n \geq 0}$ be an admissible sequence. We must show that $\hb$ contains a function in $H^*_1(M)$. If $b$ vanishes at some point of $\D$, then the implication is trivial. Assume therefore that $\log \Delta_b$ is integrable on some (say, open) interval $I$ which is not all of $\T$, and let $w = \Delta^2_b|I$ be the restriction of $\Delta^2_b$ to the interval $I$. By \thref{PowerAdmSeq} and \thref{AdmissibleSequenceLemma} there exists a function $G$ which satisfies \eqref{ExpDecTag}, \eqref{LogLogIntTag}, with corresponding moment sequence \[ \widetilde{M} = \{\widetilde{M}_n\}_{n \geq 0} = \{M_n(G)\}_{n \geq 0}\] satisfying \[ \widetilde{M}_n \leq M^2_n\] for large $n$. If \[d\mu(z) = G(1-|z|)dA(z) + w(z)d\m(z), \] then the space $\Po^2(\mu)$ is irreducible by \thref{IrrDef}, since $\co (w)$ coincides with $I$, which is a carrier of $w$. By irreducibility we have that $\L^2(w\, d\m) \not\subset \Po^2(\mu)$. So there must exist a non-zero element $F - g \in \L^2(\mu)$, with $F$ being an analytic function on $\D$ and $g$ living on $I \subset \T$, which is orthogonal to $\Po^2(\mu)$ in $\L^2(\mu)$. We can't have $g \equiv 0$, for then the Taylor coefficients of $F$ would all vanish by the orthogonality to analytic monomials, and consequently $F-g$ would reduce to the zero element. The orthogonality means that \[ F_n \widetilde{M}_{n}  = (w g)_n, \quad n \geq 0\] where $\{F_n\}_{n \geq 0}$ are the Taylor coefficients of $F$ and $(w g)_n$ are the non-negative Fourier coefficients of $w g$. For large $n$, we have the estimate \begin{align*}
|(w g)_n|^2 & = |F_n \widetilde{M}_{n}|^2 \\ &\leq  \widetilde{M}_{n} \sum_{m \geq 0} |F_m|^2 \widetilde{M}_{m} \\ & = \widetilde{M}_{n} \|F\|^2_{\mu_\D} \\
&\leq M_n^2 \|F\|^2_{\mu_\D}.
\end{align*} 
Thus $P_+wg$ is a member of $H^*_1(M)$. Since $g$ lives on $I$ and $g\Delta_b \in \L^2(I)$, we have by \thref{CauchyTransformsinHb} that $P_+w g = P_+ \Delta_b \Delta_b g\in \hb$. This function is non-zero since $\log( w |g| ) \not\in \L^1(\T)$ by the choice of $I \neq \T$. Thus $(iii) \Rightarrow (ii)$, and we have completed our proof of \thref{MainTheoremHbExistenceESD}. 

\section{Density in $\hb$ of functions with rapid spectral decay}

\label{DensityESDHbSec}

The main result of \cite{limani2023problem} characterizes the density of the functions in $\hb$ which have Taylor series $f(z) = \sum_{n \geq 0} f_nz^n$ satisfying $|f_n| = \mathcal{O}(1/n^k)$, for positive $k$. The characterization is in terms of the structure of $M_z$-invariant subspaces of $\Po^2(\mu)$ with $\mu$ of form \eqref{MuGeneralStructureEq} and $G(t) = t^k$, $k \geq 0$. The proofs in \cite{limani2023problem} in fact carry over more-or-less verbatim from the case considered there to many other function classes defined by their spectral size, with the family of functions defined by conditions such as \eqref{RapidDecayEq} being no exception. Thus, in fact, \thref{MainTheoremHbDensityESD} is more or less a direct consequence of \thref{IrrDef}, \thref{CyclicityMainTheorem} and \thref{PermanenceMainTheorem}. For reasons of completeness of the present work, we outline an argument which is in parts new, leads to a proof of \thref{MainTheoremHbDensityESD}, but also gives additional bits of information regarding which functions in $\hb$ lie outside of the closure of functions satisfying spectral decay properties as in \eqref{ESDpropEq}.

As before, $\Delta_b(x) = \sqrt{1-|b(x)|}$ for $x \in \T$, and we let \[ b = BS_\nu b_0\] be the inner-outer factorization of $b$, with $B$ a Blaschke product, $S_\nu$ a singular inner function, and $b_0$ an outer function. We denote by $I_b = BS_\nu$ the inner factor of $b$. 

\begin{lem} \thlabel{ResVanishingLemma} Let $w \in \L^1(\T)$ be non-negative, and assume that for some $g \in \L^2(w \, d\m)$ the function $P_+ w g$ satisfies \eqref{ESDpropEq}. Then $gw$ vanishes on $\res (w)$.    
\end{lem}
\begin{proof}
We use \thref{MhstarContainmentFromRSD} to obtain a measure $\mu$ as in \eqref{T1} of Section \ref{introsec}, and with the parameters $\beta = 1$ and $c > 0$ chosen so that if $M = \{M_n\}_{n \geq 0}$ is the sequence of moments corresponding to $\mu_\D$, then $P_+wg \in H^*_2(M)$. Let $h$ be a bounded function living on $\res (w)$, and $\{p_n\}_{n \geq 0}$ be a sequence of analytic polynomials which converges to $h$ in the norm of $\Po^2(\mu)$. This is possible by \thref{ResSetMainLemma}. In particular, this convergence implies that $p_n \to 0$ in $\Po^2(\mu_\D)$, or in other words, $p_n \to 0$ in $H_2(M)$. Simultaneously, we have that $p_n \to h$ in $\L^2(w\, d\m)$. Using the duality pairing \eqref{CauchyDualityHM}, we obtain \begin{align*}
    0  & = \ip{0}{P_+ wg} \\ & = \lim_{n \to \infty} \ip{p_n}{P_+wg} \\ & = \lim_{n \to \infty} \ip{p_n}{P_+ wg}_{\L^2} \\ & = \lim_{n \to \infty} \ip{p_n}{wg}_{\L^2} \\ & = \int_\T h\conj{g}\, w\, d\m.
\end{align*} Since $h$ is an arbitrary bounded function living on $\res( w)$, it follows that $gw \equiv 0$ on $\res (w)$.
\end{proof} 

\begin{prop} \thlabel{Prop1TheoremD}

Assume that the set $\res (\Delta_b)$ has positive $m$-measure, and let $s \in \L^2(\T)$ be a non-zero function which vanishes outside of $\res (\Delta_b)$. 
Then the non-zero function \[f = P_+\Delta_b s \in \hb\] lies outside of the norm-closure in $\hb$ of functions satisfying \eqref{ESDpropEq}.
\end{prop}

\begin{proof}
    Seeking a contradiction, assume that $\{h_n\}_n$ is a sequence of functions in $\hb$ which satisfy \eqref{ESDpropEq} and which converge in the norm of $\hb$ to the given $f$. In the notation of \thref{normformula}, we consider $Jh_n = (h_n, k_n) \in H^2 \oplus \L^2(E)$ and $Jf = (f,g) \in H^2 \oplus \L^2(E)$, where $g = -\conj{b}s$ according to \thref{CauchyTransformsinHb}. By \thref{TconjbInvariance}, $T_{\conj{b}}h_n$ converges to $T_{\conj{b}}f$ in the norm of $\hb$, and since the embedding $J$ of \thref{normformula} is an isometry, \thref{TconjbInvariance} moreover implies that $\conj{b}k_n$ converges to $\conj{b}g = -\conj{b^2} s$ in $\L^2(\T)$. In particular, this implies that $k_n$ cannot all simultaneously vanish on $\res (\Delta_b)$, since $s$ lives only on that set. But $T_{\conj{b}}h_n$ satisfies \eqref{ESDpropEq} (since $h_n$ does), and $T_{\conj{b}}h_n = P_+ \conj{b} h_n = P_+ \Delta_b \conj{b} k_n$ by \thref{TconjbInvariance}. Thus by \thref{ResVanishingLemma}, the functions $\Delta_b \conj{b} k_n$ must vanish on $\res ( \Delta_b)$, and consequently $k_n$ must vanish on $\res (\Delta_b)$, since $\conj{b}\Delta_b$ is non-zero $m$-almost everywhere on that set. This is the desired contradiction.
\end{proof}

We have now proved that it is necessary for $\co ( \Delta_b )$ to be a carrier for $\Delta_b$ if functions satisfying \eqref{ESDpropEq} are to be dense in $\hb$. In the next proposition, we assume that $\co (\Delta_b)$ is a carrier for $\Delta_b$, and show that if $S_\nu$ is the singular inner factor of $b$ and $\nu$ places some portion of its mass outside of the core of $\Delta_b$, then again functions satisfying \eqref{ESDpropEq} are not dense in $\hb$. And again, we do it by exhibiting explicit functions in $\hb$ which cannot be approximated in this way.

\begin{prop}\thlabel{Prop2TheoremD}  Assume that $\co ( \Delta_b)$ is a carrier for $\Delta_b$ and that \[\nu \big(\T \setminus \co (\Delta_b ) \big) > 0,\] where $S_\nu$ is the singular inner factor of $b$. Decompose the measure $\nu$ as \[\nu = \nu_r + \nu_c,\] where $\nu_r$ is the restriction of $\nu$ to the set $\T \setminus \co (\Delta_b )$, and $\nu_c$ is the restriction of $\nu$ to the set $\co (\Delta_b)$. Then all functions in the subspace \[ (b/S_{\nu_r})\K_{S_{\nu_r}} = B S_{\nu_c}b_0 \K_{S_{\nu_r}} \subset \hb\] are orthogonal in $\hb$ to functions satisfying \eqref{ESDpropEq}, $\K_{S_{\nu_r}}$ being the model space generated by the singular inner function $S_{\nu_r}$.
\end{prop}

\begin{proof}
Take a function $f = BS_{\nu_c} b_0  s$, where $s \in \K_{S_{\nu_r}}$, and $h$ satisfying \eqref{ESDpropEq}. In the notation of \thref{normformula}, a computation shows that $Jf = (f, g)$, where \[ g = \Delta_b \conj{S_{\nu_r}} s. \] Let $\Po^2(\mu)$ and $H_2(M) = \Po^2(\mu_\D)$ be as in the proof of \thref{ResVanishingLemma}, with $w = \Delta_b$ and the sequence $M$ being chosen so that $h \in H^*_2(M)$. This time the space $\Po^2(\mu)$ is irreducible, and by \thref{CyclicityMainTheorem} the singular inner function $S_{\nu_r}$ is cyclic in $\Po^2(\mu)$. Hence there exists a sequence of analytic polynomials $\{p_n\}_n$ such that $S_{\nu_r} p_n$ converges to the function $s \in H^2$ in the norm of $\Po^2(\mu)$, and in particular in the norm of $H_2(M)$. Multiplying this sequence by $B S_{\nu_c} b_0$, it follows from \thref{HToeplitzInvariance} that $bp_n$ converges to $f$ in $H_2(M)$. Simultaneously, the $\Po^2(\mu)$-convergence implies that $S_{\nu_r}p_n$ converge to $s$ in $\L^2( \Delta_b d\m)$, and since $S_{\nu_r}$ is unimodular on $\T$, in fact we have that $p_n$ converge to $\conj{S_{\nu_r}}s$ in $\L^2(\Delta_b d\m)$. Let $Jh = (h,k)$ be the corresponding pair for $h$. We can use the above claims to compute  

\begin{align*}
    \ip{h}{f}_{\L^2} + \ip{k}{g}_{\L^2} & = \ip{h}{f} + \ip{k}{\Delta_b \conj{S_{\nu_r}}s}_{\L^2} \\ & = \lim_{n \to \infty} \ip{h}{b p_n} + \ip{k}{\Delta_b p_n}_{\L^2} \\
    & = \lim_{n \to \infty} \ip{h}{b p_n}_{\L^2} + \ip{k}{\Delta_b p_n}_{\L^2}    
    \\ &= \lim_{n \to \infty} \ip{P_+ (\conj{b}h + \Delta_b k)}{p_n}_{\L^2} \\ & = \lim_{n \to \infty} \ip{0}{p_n}_{\L^2} \\
    & = 0.
\end{align*}
In the last step we used condition \eqref{hbconteq} for the pair $(h,k)$. Since the embedding $J$ in \thref{normformula} is an isometry, it follows that $f$ is orthogonal to $h$ in $\hb$.
\end{proof}

\begin{proof}[Proof of \thref{MainTheoremHbDensityESD}]
We see from \thref{Prop1TheoremD} and \thref{Prop2TheoremD} above that condition $(iii)$ in \thref{MainTheoremHbDensityESD} is necessary in order for $(i)$ to hold. Since $(ii)$ implies $(i)$, it suffices thus to show that $(iii)$ implies $(ii)$. The argument is essentially same as the one appearing in \cite{limani2023problem} and \cite{limani_malman_2023}, we include it only for completeness.

Just as in the proof of \thref{MainTheoremHbExistenceESD}, given an admissible sequence $M = \{M_n\}_{n \geq 0}$ we use \thref{PowerAdmSeq} and \thref{AdmissibleSequenceLemma} to obtain $G$ satisfying \eqref{ExpDecTag}, \eqref{LogLogIntTag}, with moment sequence $\widetilde{M} = \{\widetilde{M}_n\}_{n \geq 0}$ satisfying $\widetilde{M}_n \leq M^2_n$ for large $n$. We must show that $\hb \cap H^*_1(M)$ is dense in $\hb$. By \thref{H1starH2starEmbeddingLemma} it will suffice to show that $H^*_2(M^2) \cap \hb$ is dense in $\hb$.

The space $\Po^2(\mu)$ constructed from the measure \[ d\mu(z) = G(1-|z|)dA(z) + \Delta_b^2(z)d\m(z)\] is irreducible by \thref{IrrDef}. Let us assume that $f \in \hb$ is orthogonal to $H^*_2(M^2) \cap \hb$. We will show that $f = 0$, which will prove \thref{MainTheoremHbDensityESD}. Because the mapping $J$ in \thref{normformula} is an isometry, it follows that $Jf$ is orthogonal to $J(H^*_2(M^2) \cap \h(b))$. Note that $J(H^*_2(M^2) \cap \h(b))$ is a subset of $H^*_2(M^2) \oplus \L^2(E)$, and under the duality pairing \eqref{CauchyDualityHM} between $H_2(M^2)$ and $H^*_2(M^2)$, we have \begin{equation}
\label{preanneq} J(H^*_2(M^2) \cap \h(b)) = \cap_{h \in H^2} \ker l_h,
\end{equation} where $l_h$ is the functional on $H^*_2(M^2) \oplus \L^2(E)$ which acts by the formula \[l_h(f,g) := \ip{f}{bh} + \ip{g}{\Delta_b h}_{\L^2}.\] This follows readily from \thref{normformula} (see, for instance, the argument in \cite{limani2023problem}). The fact that $Jf$ annihilates $J(H^*_2(M^2) \cap \h(b))$ and that \eqref{preanneq} holds implies that $Jf$ is contained in the weak-star closure of the linear manifold $\{l_h\}_{H^2} \subseteq H_2(M^2) \oplus \L^2(E)$. Since the pairing between $H_2(M^2)$ and $H^*_2(M^2)$ is reflexive and $\{l_h\}_{h \in H^2}$ is a convex set, basic functional analysis says that, in fact, $Jf$ is contained in the norm-closure of $\{l_h\}_{h \in H^2}$.  Thus there exists a sequence $\{h_k\}_{k \geq 1}$ with $h_k \in H^2$ such that \begin{equation}
\label{seqconv} (bh_k, \Delta_b h_k) \to Jf := (f,g)
\end{equation} in the norm of $H_2(M^2) \oplus \L^2(E)$. Multiply the second coordinate by $b$ to obtain 
\begin{equation}
\label{seqconv2} (bh_k, \Delta_b b h_k) \to (f,bg).
\end{equation}
But the inequalities $\widetilde{M}_n \leq M^2_n$ imply that $bh_k$ converges to $f$ also in the space $\Po^2(\mu_\D)$, and so in fact \eqref{seqconv2} tells us that $\{bh_k\}_n$ is a Cauchy sequence in $\Po^2(\mu)$, to which \thref{PermanenceMainTheorem} applies. The critical conclusion is that $bh_k \to f$ in the irreducible $\Po^2(\mu)$. If $I_b$ is the inner factor of $b$, then \thref{PermanenceMainTheorem} implies that $f/I_b \in H^2$, and by the irreducibility of $\Po^2(\mu)$ the sequence $bh_k$ on $\T$ must converge to the boundary function of $f$ on $\T$. Thus $g = \Delta_b f/b$ by \eqref{seqconv2}, and $Jf = (f, \Delta_b f/b)$. By \thref{normformula} we get that \begin{equation} \label{projzero}
0 = P_+(\conj{b}f + \Delta_b g) = P_+(\conj{b}f + \Delta_b^2 f/b) = P_+(|b|^2f/b + \Delta_b^2 f/b) = P_+(f/b).
\end{equation} From the above computation we infer that, in terms of boundary values, we have $f/b = \conj{b}f + \Delta_b g \in \L^2(\T)$, and consequently $f/b$ has square-integrable boundary values. Since $f/I_b \in H^2$, it follows from the classical Smirnov maximum principle that $f/b \in H^2$. Then $f/b$ is an analytic function which projects to $0$ under $P_+$, which implies that $f/b = 0$, and consequently $f=0$.
\end{proof}

\section{Proof of \thref{UncertThmRSD}}

\label{ThmCProofSec}

A proof of \thref{UncertThmRSD} relies on a judicious application of \thref{ResVanishingLemma}. 

\begin{proof}[Proof of \thref{UncertThmRSD}]
If $C_\nu$ satisfies \eqref{ESDpropEq}, then the function $f(z) = \sum_{n \geq 0} \nu_n z^n$, $z \in \T$, is certainly smooth on $\T$ and it has an analytic extension to $\D$. Since the Cauchy transform of the measure $d\nu - f \, d\m$ vanishes in $\D$, this measure must be absolutely continuous with respect to $m$ by the classical theorem of brothers Riesz. Hence $d\nu$ is also absolutely continuous. Let $g \in \L^1(\T)$ be its Radon-Nikodym derivative, so that $d\nu = g \, d\m$. Set $f = \C_\nu = \C_g$, which by our assumption is a function satisfying \eqref{ESDpropEq}. Unfortunately, we cannot directly apply \thref{ResVanishingLemma} since we do not necessarily have that $g \in \L^2(\T)$. We must take care of this slight inconvenience to prove the theorem. 

\thref{ToeplitzInvRSD} says that $T_{\conj{h}}f$ satisfies \eqref{ESDpropEq}, where $T_{\conj{h}}$ is any co-analytic Toeplitz operator with bounded symbol $h \in H^\infty$. Moreover, $T_{\conj{h}} f$ has the representation \[ T_{\conj{h}}f(z) = \C_{\conj{h}g}(z), \quad z \in \D.\] The above formula can be derived by first showing through simple algebraic manipulations that it holds for $h(z) := z$, then for analytic monomials by iteration, and thus for analytic polynomials by linearity. Finally, fix a uniformly bounded sequence of analytic polynomials $\{p_n\}_{n \geq 1}$ which converges to $h$ pointwise $m$-almost everywhere on $\T$ (the polynomials $p_n$ could be taken to be the Ces\`aro means of the partial sums of the Taylor series of $h$). For such a sequence we readily see from the dominated convergence theorem that \[ T_{\conj{h}}f(z) = \lim_{n \to \infty} T_{\conj{p_n}}f(z) = \lim_{n \to \infty} \C_{\conj{p_n}g}(z) = \C_{\conj{h}g}(z), \quad z \in \D.\] Since $g \in L^1(\T)$, in particular we have that $\log^+|g| \in L^1(\T)$, and this means that an outer function $h \in H^\infty$ exists which satisfies the boundary value equation \[ |h(x)| = \min\Big( 1, 1/|g(x)| \Big)\] for $\m$-almost every $x \in \T$. Set also \[w(x) = \min\Big( 1, |g(x)|\Big).\]
Now, we can write \[\conj{h} g = \conj{h} \frac{g}{w} w = u w\] with \[ u := \conj{h} \frac{g}{w}.\] It is easily checked that $u$ satisfies $|u(x)| = 1$ for $\m$-almost every $x$ for which $|g(x)| > 0$. Then \[ T_{\conj{h}}f(z) = \C_{\conj{h} g}(z) = P_+ u w(z), \quad z \in \D\] and \thref{ResVanishingLemma} can be applied to conclude that $uw$ vanishes on $\res (w)$. Since $u$ is unimodular, it follows that in fact $w$ vanishes on $\res (w)$, and consequently the set \[ \{ x \in \T : w(x) > 0 \} = \{ x \in \T : |g(x)| > 0 \}\] coincides with $\co (w)$, up to a set of $\m$-measure zero. For any interval $I$ contained in $\co (w)$ it follows from the pointwise inequality $|g| \geq w$ and the definition of $\co (w)$ that  \[ \int_I \log |g| \, dm \geq \int_I \log w \,d\m > -\infty.\] Thus $g$ has structure as claimed in the statement of \thref{UncertThmRSD}, and the proof is complete.
\end{proof}

\bibliographystyle{plain}
\bibliography{mybib}

\end{document}